\documentclass[11pt]{amsart}
\usepackage{latexsym}
\usepackage{amssymb}
\usepackage{amsfonts}
\usepackage{graphicx}
\usepackage{epsf}
\usepackage[all]{xypic}
\usepackage{delarray}
\usepackage{setspace}
\newtheorem{theorem}{Theorem}[section]
\newtheorem{corollary}[theorem]{Corollary}
\newtheorem{lemma}[theorem]{Lemma}
\newtheorem{proposition}[theorem]{Proposition}

\newtheorem{example}[theorem]{Example}

\definemorphism{dato}\dashed \tip \notip
\definemorphism{Dato}\Dashed \tip \notip

\newcommand{\Ker}{{\rm Ker}\,}

\newcommand{\Mm}{\mathcal{M}}

\def\CC{{\mathbb C}}
\def\ZZ{{\mathbb Z}}

\begin{document}

\sloppy

\title[Quiver algebras and coalgebras]{Quiver Algebras, Path Coalgebras and co-reflexivity}


\dedicatory{\it Dedicated to our friend Margaret Beattie on the
occasion of her retirement}

\begin{abstract}
We study the connection between two combinatorial notions associated
to a quiver: the quiver algebra and the path coalgebra. We show that
the quiver coalgebra can be recovered from the quiver algebra as a
certain type of finite dual, and we show precisely when the path
coalgebra is the classical finite dual of the quiver algebra, and
when all finite dimensional quiver representations arise as
comodules over the path coalgebra. We discuss when the quiver
algebra can be recovered as the rational part of the dual of the
path coalgebra. Similar results are obtained for incidence
(co)algebras. We also study connections to the notion of coreflexive
(co)algebras, and give  a partial answer to an open problem
concerning tensor products of coreflexive coalgebras.
\end{abstract}

\author{S.D\u{a}sc\u{a}lescu${}^1$, M.C. Iovanov${}^{1,2}$, C. N\u{a}st\u{a}sescu${}^1$}
\address{${}^1$University of Bucharest, Facultatea de Matematica\\ Str.
Academiei 14, Bucharest 1, RO-010014, Romania}
\address{${}^2$University of Southern California \\
3620 S Vermont Ave, KAP 108 \\
Los Angeles, CA 90089, USA}
\address{e-mail: sdascal@fmi.unibuc.ro, iovanov@usc.edu, yovanov@gmail.com} \address{Constantin\_nastasescu@yahoo.com}

\date{}

\maketitle

\section{Introduction and Preliminaries}

Let $\Gamma$ be a quiver, and let $K$ be an arbitrary ground field,
which will be fixed throughout the paper. The associated quiver
algebra $K[\Gamma]$ is an important object studied extensively in
representation theory, and one theme in the field is to relate and
understand combinatorial properties of the quiver via the properties
of the category of representations of the quiver, and vice versa.
Quiver algebras also play a role in general representation theory of
algebras; for example, every finite dimensional pointed algebra is a
quiver algebra ``with relations". A closely related object is the
path coalgebra $K\Gamma$, introduced in \cite{cm}, together with its
comodules (quiver corepresentations). Comodules over path coalgebras
turn out to form a special kind of representations of the quiver,
called locally nilpotent representations in \cite{chin}. A natural
question arises then: what is the precise connection between the two
objects $K[\Gamma]$ and $K\Gamma$.  We aim to provide such
connections, by finding out when one of these objects can be
recovered from the other one. This is also important from the
following viewpoint: one can ask when the finite dimensional locally
nilpotent representations of the quiver (i.e., quiver
corepresentations), provide all the finite dimensional quiver
representations. This situation will be exactly the one in which the
path coalgebra is recovered from the quiver algebra by a certain
natural construction involving representative functions, which we
recall below.

Given a coalgebra $C$, its dual $C^*$ is always an algebra. Given an
algebra $A$, one can associate a certain subspace $A^0$ of the dual
$A^*$, which  has a coalgebra structure. This is called the finite
dual of $A$, and it plays an important role in the representation
theory of $A$, since the category of locally finite left $A$-modules
(i.e., modules which are sums of their finite dimensional
submodules) is isomorphic to the category of right $A^0$-comodules
(see, for example, \cite{G}). $A^0$ is sometimes also called the
coalgebra of representative functions, and consists of all
$f:A\rightarrow K$ whose kernel contains a cofinite (i.e., having
finite codimension) ideal. We show that the path coalgebra $K\Gamma$
can be reconstructed from the quiver algebra $K[\Gamma]$ as a
certain type of ``graded" finite dual, that is, $K\Gamma$ embeds in
the dual space $K[\Gamma]^*$ as the subspace of linear functions
$f:K[\Gamma]\rightarrow K$ whose kernel contains a cofinite monomial
ideal. This is an ``elementwise" answer to the recovery problem; its
categorical analogue states that the comodules over the quiver
coalgebra are precisely those quiver representations in which the
annihilator of every element contains a cofinite monomial ideal. In
order to connect these to the classical categorical duality, we
first note that
 in general the quiver algebra does not have identity, but it has enough idempotents. Therefore,
we first extend the construction of the finite dual to algebras with
enough idempotents (Section \ref{sectdualfinit}). To such an algebra
$A$ we associate a coalgebra $A^0$ with counit, and we show that the
category of right $A^0$-comodules is isomorphic to the category of
locally finite unital $A$-modules. In Section \ref{sectquiverpath}
we show that the path coalgebra $K\Gamma$ embeds in $K[\Gamma]^0$,
and we prove that this embedding is an isomorphism, i.e., the path
coalgebra can be recovered as the finite dual of the quiver algebra,
if and only if the quiver has no oriented cycles and there are
finitely many arrows between any two vertices. On the other hand,
$K[\Gamma]$ embeds as an algebra without identity in the dual
algebra $(K\Gamma)^*$ of the path coalgebra. We show that the image
of this embedding is the rational (left or right) part of
$(K\Gamma)^*$, i.e., the quiver algebra can be recovered as the
rational part of the dual of the path coalgebra, if and only if for
any vertex $v$ of $\Gamma$ there are finitely many paths starting at
$v$ and finitely many paths ending at $v$. This is also equivalent
to the fact that $K\Gamma$ is a left and right semiperfect
coalgebra.

In Section \ref{sectincidence} we obtain similar results for another
class of (co)algebras which are also objects of great combinatorial
interest, namely for incidence (co)algebras. See \cite{jr}, for
instance. We show that the incidence coalgebra of a partially
ordered set $X$ is always the finite dual of a subalgebra $FIA(X)$
of the incidence algebra which consists of functions of finite
support. In this setting, this algebra $FIA(X)$ is the natural
analogue of the quiver algebra.

 It is also interesting
to know when can $K\Gamma$ be recovered from $(K\Gamma)^*$, and how
this relates to the results of Section \ref{sectquiverpath}.  This
problem is related to an important notion in coalgebra theory, that
of coreflexive coalgebra. A coalgebra $C$ over $K$ is coreflexive if
the natural coalgebra embedding $C\rightarrow (C^*)^0$ is an
isomorphism. In other words, $C$ is coreflexive if it can be
completely recovered from its dual. In Section \ref{s5} we aim to
study this condition for path coalgebras and their subcoalgebras,
and give the connection with the results of Section
\ref{sectquiverpath}. We show that, in fact, a path coalgebra of a
quiver with no loops and finitely many arrows between any two
vertices is not necessarily coreflexive, and also, that the quivers
of coreflexive path coalgebras can contain loops. We then prove a
general result stating that under certain conditions a coalgebra $C$
is coreflexive if and only if its coradical is coreflexive. In
particular, this result holds for subcoalgebras of a path coalgebra
$K\Gamma$ with the property that there are finitely many paths
between any two vertices of $\Gamma$. The result  applies in
particular to incidence coalgebras. For both a path coalgebra and an
incidence coalgebra  the coradical is a grouplike coalgebra (over
the set of vertices of the quiver for the first one, or the
underlying ordered set for the second one). Thus the coreflexivity
of such a coalgebra reduces to the coreflexivity of a grouplike
coalgebra $K^{(X)}$. By the results of Heyneman and Radford
\cite[Theorem 3.7.3]{HR},  if $K$ is an infinite field, then
$K^{(X)}$ is coreflexive for most sets in $X$ in set theory and any
set of practical use (see Section 5).

We use our results to give a partial answer to a question of
E.J.~Taft and D.E.~Radford asking whether the tensor product of two
coreflexive coalgebras is coreflexive. In particular, we show that
the tensor product of two coreflexive pointed coalgebras, which
embed in path coalgebras of quivers with only finitely many paths
between any two vertices, is coreflexive.

Throughout the paper $\Gamma=(\Gamma_0,\Gamma_1)$ will be a quiver.
$\Gamma_0$ is the set of vertices, and $\Gamma_1$ is the set of
arrows of $\Gamma$. If $a$ is an arrow from the vertex $u$ to the
vertex $v$, we denote $s(a)=u$ and $t(a)=v$. A path in $\Gamma$ is a
finite sequence of arrows $p=a_1a_2\ldots a_n$, where $n\geq 1$,
such that $t(a_i)=s(a_{i+1})$ for any $1\leq i\leq n-1$. We will
write $s(p)=s(a_1)$ and $t(p)=t(a_n)$. Also the length of such a $p$
is $n$. Vertices $v$ in $\Gamma_0$ are also considered as paths of
length zero, and we write $s(v)=t(v)=v$. If $p$ and $q$ are two
paths such that $t(p)=s(q)$, we consider the path $pq$ by taking the
arrows of $p$ followed by the arrows of $q$. We denote by $K\Gamma$
the path coalgebra, which is the vector space with a basis
consisting of all paths in $\Gamma$, comultiplication $\Delta$
defined by $\Delta(p)=\sum _{qr=p}q\otimes r$ for any path $p$, and
counit $\epsilon$ defined by $\epsilon(v)=1$ for any vertex $v$, and
$\epsilon(p)=0$ for any path of positive length. The underlying
space of $K\Gamma$ can be also endowed with a structure of an
algebra, not necessarily with identity, with the multiplication
defined such that the product of two paths $p$ and $q$ is $pq$ if
$t(p)=s(q)$, and 0 otherwise. We denote this algebra by $K[\Gamma]$;
this is known in literature as the quiver algebra or the path
algebra of $\Gamma$. It has identity if and only if $\Gamma_0$ is
finite, and in this case the sum of all vertices is the identity.

Besides the above mentioned recovery connections between quiver
algebras and path coalgebras, one can also ask whether there is any
compatibility between them. More precisely, when do the two
structures on the same vector space $K\Gamma$ give rise to a
bialgebra structure. This turns out to be only the case for very
special quivers. Specifically, consider $K[\Gamma]$ to be the vector
space with basis the oriented paths of $\Gamma$, and with the quiver
algebra and path coalgebra structures. Then $K[\Gamma]$ is a
bialgebra (with enough idempotents in general) if and only if in
$\Gamma$ there are no (directed) paths of length $\geq 2$ and no
multiple edges between vertices (i.e. for any two vertices $a,b$ of
$\Gamma$ there is at most one edge from $a$ to $b$). Indeed,
straightforward computations show that whenever multiple edges
$\bullet\stackrel{x,y}{\Longrightarrow}\bullet$ or paths
$\bullet\stackrel{x}{\longrightarrow}\bullet
\stackrel{y}{\longrightarrow}\bullet$ of length at least 2 occur,
then $\Delta(xy)\neq \Delta(x)\Delta(y)$. Conversely, a case by case
computation for $\Delta(pq)$ with $p,q$ paths of possible length $0$
or $1$ will show that $\Delta(pq)=\Delta(p)\Delta(q)$.

This shows that the relation between the path coalgebra and quiver
algebra is more of a dual nature than an algebraic compatibility.
For basic terminology and notation about coalgebras and comodules we
refer to \cite{DNR}, \cite{mo}, and \cite{sw}. All (co)algebras and
(co)modules considered here will be vector spaces over $K$, and
duality $(-)^*$ represents the dual $K$-vector space.

\section{The finite dual of an algebra with enough idempotents}\label{sectdualfinit}

In this section we extend the construction of the finite dual of an
algebra with identity to the case where $A$ does not necessarily
have a unit, but it has enough idempotents. Throughout this section
we consider a $K$-algebra $A$, not necessarily having a unit, but
having a system $(e_{\alpha})_{\alpha \in R}$ of pairwise orthogonal
 idempotents, such that $A=\oplus_{\alpha \in R}Ae_{\alpha}=\oplus_{\alpha
\in R}e_{\alpha}A$. Such an algebra is said to have ``enough
idempotents", and it is also called an algebra with a complete
system of orthogonal idempotents in the literature. Let us note that
$A$ has local units, i.e., if $a_1,\ldots ,a_n\in A$, then there
exists an idempotent $e\in A$ (which can be taken to be the sum of
some $e_{\alpha}$'s) such that $ea_i=a_ie=a_i$ for any $1\leq i\leq
n$. Our aim is to show that there exists a natural structure of a
coalgebra (with counit) on the space
$$A^0=\{ f\in A^*|\; \Ker(f)\; \mbox{contains an ideal of }A\mbox{ of finite
codimension}\}.$$ We will call $A^0$ the finite dual of the algebra
$A$.

\begin{lemma} \label{lemaidealcofinit}
Let $I$ be an ideal of $A$ of finite codimension. Then the set
$R'=\{\alpha \in R|\; e_{\alpha}\notin I\}$ is finite.
\end{lemma}
\begin{proof}
Denote by $\hat{a}$ the class of an element $a\in A$ in the factor
space $A/I$. We have that $(\widehat{e_{\alpha}})_{\alpha\in R'}$ is
linearly independent in $A/I$. Indeed, if $\sum_{\alpha\in
R'}\lambda_{\alpha}\widehat{e_{\alpha}}=0$, then $\sum_{\alpha\in
R'}\lambda_{\alpha}e_{\alpha}\in I$. Multiplying by some
$e_{\alpha}$ with $\alpha \in R'$, we find that
$\lambda_{\alpha}e_{\alpha}\in I$, so then necessarily
$\lambda_{\alpha}=0$. Since $A/I$ is finite dimensional, the set
$R'$ must be finite.
\end{proof}

Assume now that $B$ is another algebra with enough idempotents, say
that $(f_{\beta})_{\beta \in S}$ is a system of orthogonal
idempotents in $B$ such that $B=\oplus_{\beta \in
S}Bf_{\beta}=\oplus_{\beta \in S}f_{\beta}B$.

\begin{lemma}\label{lemaIJ}
Let $H$ be an ideal of $A\otimes B$ of finite codimension. Let $I=\{
a\in A|\; a\otimes B\subseteq H\}$ and $J=\{ b\in B|\; A\otimes
b\subseteq H\}$. Then $I$ is an ideal of $A$ of finite codimension,
$J$ is an ideal of $B$ of finite codimension and $I\otimes
B+A\otimes J\subseteq H$.
\end{lemma}
\begin{proof}
Let $a\in I$ and $a'\in A$. If $b\in B$ and $f$ is an idempotent in
$B$ such that $fb=b$, we have that $a'a\otimes b=a'a\otimes
fb=(a'\otimes f)(a\otimes b)\in H$. Thus $a'a\otimes B\subseteq H$,
so $a'a\in I$. Similarly $aa'\in I$, showing that $I$ is an ideal of
$A$. Similarly $J$ is an ideal of $B$.

It is clear that $(e_{\alpha}\otimes f_{\beta})_{\alpha\in
R,\beta\in S}$ is a set of orthogonal idempotents in $A\otimes B$
and $A\otimes B=\oplus_{\alpha \in R,\beta \in S}(A\otimes
B)(e_{\alpha}\otimes f_{\beta})=\oplus_{\alpha \in R,\beta \in
S}(e_{\alpha}\otimes f_{\beta})(A\otimes B)$. By Lemma
\ref{lemaidealcofinit}, there are finitely many idempotents
$e_{\alpha_1}\otimes f_{\beta_1},\ldots,e_{\alpha_n}\otimes
f_{\beta_n}$ which lie outside $H$. If $\alpha \in R\setminus
\{\alpha_1,\ldots,\alpha_n\}$, then for any $\beta \in S$ we have
that $e_{\alpha}\otimes f_{\beta}\in H$, so $e_{\alpha}\otimes
Bf_{\beta}=(e_{\alpha}\otimes Bf_{\beta})(e_{\alpha}\otimes
f_{\beta})\subseteq H$. Then $e_{\alpha}\otimes B\subseteq H$, so
$e_{\alpha}\in I$. Similarly for any $\beta \in S\setminus
\{\beta_1,\ldots,\beta_n\}$ we have that $f_{\beta}\in J$.

For any $\beta\in S$ let $\phi_{\beta}:A\rightarrow A\otimes B$ be
the linear map defined by $\phi_{\beta}(a)=a\otimes f_{\beta}$. If
$a\in A$, then $a\in I$ if and only if for any $\beta \in S$ we have
$a\otimes Bf_{\beta}\subseteq H$; because there is a local unit for
$a$, this is further equivalent to $a\otimes f_{\beta}\in H$ for
$\beta\in S$. This condition is obviously satisfied for $\beta \in
S\setminus \{\beta_1,\ldots,\beta_n\}$ since $f_\beta \in J$, so we
obtain that $I=\cap_{1\leq i\leq n}\phi_{\beta_i}^{-1}(H)$, a finite
intersection of finite codimensional subspaces of $A$, thus a finite
codimensional subspace itself. Similarly $J$ has finite codimension
in $B$. The fact that $I\otimes B+A\otimes J\subseteq H$ is obvious.
\end{proof}

Now we essentially proceed as in \cite[Chapter VI]{sw} or
\cite[Section 1.5]{DNR}, with some arguments adapted to the case of
enough idempotents.

\begin{lemma}\label{lemaAotB}
Let $A$ and $B$ be algebras with enough idempotents. The
following assertions hold.\\
(i) If $f:A\rightarrow B$ is a morphism of algebras, then
$f^*(B^0)\subseteq A^0$, where $f^*$ is the dual map of $f$.\\
(ii) If $\phi:A^*\otimes B^*\rightarrow (A\otimes B)^*$ is the
natural linear injection, then $\phi (A^0\otimes B^0)=(A\otimes B)^0$.\\
(iii) If $M:A\otimes A\rightarrow A$ is the multiplication of $A$,
and $\psi:A^*\otimes A^*\rightarrow (A\otimes A)^*$ is the natural
injection, then $M^*(A^0)\subseteq \psi(A^0\otimes A^0)$.
\end{lemma}
\begin{proof}
It goes as the proof of \cite[Lemma 1.5.2]{DNR}, with part of the
argument in (ii) replaced by using the construction and the result
of Lemma \ref{lemaIJ}.
\end{proof}

Lemma \ref{lemaAotB} shows that by restriction and corestriction we
can regard the natural linear injection $\psi$ as an isomorphism
$\psi:A^0\otimes A^0\rightarrow (A\otimes A)^0$. We consider the map
$\Delta:A^0\rightarrow A^0\otimes A^0$, $\Delta=\psi^{-1}M^*$. Thus
$\Delta(f)=\sum_i u_i\otimes v_i$ is equivalent to
$f(xy)=\sum_iu_i(x)v_i(y)$ for any $x,y\in A$. On the other hand, we
define a linear map $\varepsilon:A^0\rightarrow K$ as follows. If
$f\in A^0$, then $\Ker(f)$ contains a finite codimensional ideal
$I$. By Lemma \ref{lemaidealcofinit}, there are finitely many
$e_{\alpha}$'s outside $I$. Therefore only finitely many
$e_{\alpha}$'s lie outside $\Ker(f)$, so it makes sense to define
$\varepsilon(f)=\sum_{\alpha \in R}f(e_{\alpha})$ (only finitely
many terms are non-zero).

\begin{proposition}
Let $A$ be an algebra with enough idempotents. Then $(A^0,\Delta,
\varepsilon)$ is a coalgebra with counit.
\end{proposition}
\begin{proof}
The proof of the coassociativity works exactly as in the case where
$A$ has a unit, see \cite[Proposition 1.5.3]{DNR}. To check the
property of the counit, let $f\in A^0$ and
$\Delta(f)=\sum_iu_i\otimes v_i$. Let $a\in A$ and $F$ a finite
subset of $R$ such that $a\in \sum_{\alpha\in F}e_{\alpha}A$. Then
clearly $(\sum_{\alpha\in F}e_{\alpha})a=a$. We have that
\begin{eqnarray*}
(\sum_i\varepsilon(u_i)v_i)(a)&=&\sum_{i,\alpha}u_i(e_{\alpha})v_i(a)\\
&=&\sum_{\alpha}f(e_{\alpha}a)\\
&=&\sum_{\alpha \in F}f(e_{\alpha}a)\\
&=&f((\sum_{\alpha \in F}e_{\alpha})a)\\
&=&f(a) \end{eqnarray*} so $\sum_i\varepsilon(u_i)v_i=f$. Similarly
$\sum_i\varepsilon(v_i)u_i=f$, and this ends the proof.
\end{proof}

Let us note that if $f:A\rightarrow B$ is a morphism of algebras
with enough idempotents, then the map $f^0:B^0\rightarrow A^0$
induced by $f^*$ is compatible with the comultiplications of $A^0$
and $B^0$, but not necessarily with the counits (unless $f$ is
compatible in some way to the systems of orthogonal idempotents in
$A$ and $B$). \\
We denote by $\rightharpoonup$ (respectively $\leftharpoonup$) the
usual left (respectively right) actions of $A$ on $A^*$. As in the
unitary case, we have the following characterization of the elements
of $A^0$.

\begin{proposition}
Let $f\in A^*$. With notation as above, the following assertions are
equivalent.\\
(1) $f\in A^0$.\\
(2) $M^*(f)\in \psi(A^0\otimes A^0)$.\\
(3) $M^*(f)\in \psi(A^*\otimes A^*)$.\\
(4) $A\rightharpoonup f$ is finite dimensional.\\
(5) $f\leftharpoonup A$ is finite dimensional.\\
(6) $A\rightharpoonup f\leftharpoonup A$ is finite dimensional.\\
(7) $\Ker (f)$ contains a left ideal of finite codimension.\\
(8) $\Ker (f)$ contains a right ideal of finite codimension.
\end{proposition}
\begin{proof}
$(1)\Rightarrow (2)\Rightarrow (3)\Rightarrow (4)$ and
$(1)\Rightarrow (6)$ work exactly as in the case where $A$ has
identity, see \cite[Proposition 1.5.6]{DNR}. We adapt the proof of
$(4)\Rightarrow (1)$ to the case of enough idempotents. Since
$A\rightharpoonup f$ is a left $A$-submodule of $A^*$, there is a
morphism of algebras (without unit) $\pi:A\rightarrow
End(A\rightharpoonup f)$ defined by $\pi (a)(m)=a\rightharpoonup m$
for any $a\in A$, $m\in A\rightharpoonup f$. Since
$End(A\rightharpoonup f)$ has finite dimension, we have that $I=\Ker
(\pi)$ is an ideal of $A$ of finite codimension. Let $a\in I$. Then
$a\rightharpoonup(b\rightharpoonup f)=(ab)\rightharpoonup f=0$ for
any $b\in A$, so $f(xab)=0$ for any $x,b\in A$. Let $e\in A$ such
that $ea=ae=a$. Then $f(a)=f(eae)=0$, so $a\in \Ker(f)$. Thus
$I\subseteq \Ker(f)$, showing that $f\in A^0$. The equivalence
$(1)\Leftrightarrow (5)$ is proved
similarly.\\
$(6)\Rightarrow (1)$ can be adapted from the unital case, see
\cite[Lemma 9.1.1]{mo}, with a small change. Indeed,
$R=(A\rightharpoonup f\leftharpoonup A)^\perp=\{ x\in
A|g(x)=0\;\mbox{for any }g\in A\rightharpoonup f\leftharpoonup A\}$
is an ideal of $A$ of finite codimension, and $R\subseteq \Ker (f)$,
since for any $r\in R$ there exists $e\in A$ such that $r=er=re$, so
then $f(r)=f(ere)=(e\rightharpoonup f\leftharpoonup e)(r)=0$.\\
$(1)\Rightarrow (7)$ is obvious, while $(7)\Rightarrow (1)$ follows
from the fact that a left ideal $I$ of finite codimension contains
the finite codimensional ideal $J=\{ r\in A|rA\subseteq I\}$.
$(1)\Leftrightarrow (8)$ is similar.
\end{proof}

We end this section with an interpretation of the connection between
an algebra $A$ with enough idempotents and its finite dual $A^0$
from the representation theory point of view. This extends the
results presented in \cite[Chapter 3, Section 1.2]{abe} in the case
where $A$ has identity. Let $M$ be a left $A$-module. Then $M$ is
called unital if $AM=M$. Also, $M$ is called locally finite if the
submodule generated by any element is finite dimensional. Denote by
$LocFin_A{\mathcal M}$ the full subcategory of the category of left
$A$-modules consisting of all locally finite unital modules. We will
also use the notations ${}_A\Mm$, $\Mm_A$ for the categories of
left, or right modules over $A$; similarly, for a coalgebra $C$,
${}^C\Mm$ and $\Mm^C$ will be used to denote the categories of left
and respectively right comodules.

\begin{proposition} \label{propequivrepcorep}
Let $A$ be an algebra with enough idempotents. Then the category
${\mathcal M}^{A^0}$ of right $A^0$-comodules is isomorphic to the
category $LocFin_A{\mathcal M}$.
\end{proposition}
\begin{proof}
Let $M$ be a right $A^0$-comodule with comodule structure $m\mapsto
\sum m_0\otimes m_1$. Then $M$ is a left $A$-module with the action
$am=\sum m_1(a)m_0$ for any $a\in A$ and $m\in M$. The counit
property $m=\sum \varepsilon (m_1)m_0$, with all $m_1$'s in $A^0$,
shows that $m=\sum _{\alpha \in F}e_{\alpha}m$ for a finite set $F$,
so $M$ is unital. Since $Am$ is contained in the subspace spanned by
all $m_0$'s, we have that $M$ is also locally finite.

Conversely, let $M\in LocFin_A{\mathcal M}$. Let $m\in M$, and let
$(m_i)_{i=1,n}$ be a (finite) basis of $Am$. Define
$a_1^*,\ldots,a_n^*\in A^*$ such that $am=\sum_{i=1,n}a_i^*(a)m_i$
for any $a\in A$. Since $\cap_{i=1,n}ann_A(m_i)=ann_A(Am)\subseteq
ann_A(m)=\cap_{i=1,n}\Ker a_i^*$ and each $ann_A(m_i)$ has finite
codimension, we get that $a_i^*\in A^0$ for any $i$. Now we define
$\rho:M\rightarrow M\otimes A^0$ by $\rho (m)=\sum_{i=1,n}m_i\otimes
a_i^*$. It is easy to see that the definition of $\rho(m)$ does not
depend on the choice of the basis $(m_i)_i$, and that $(\rho \otimes
I)\rho=(I\otimes \Delta )\rho$. To show that $M$ is a right
$A^0$-comodule it remains to check the counit property, and this
follows from the fact that $M$ is unital.

It is clear that the above correspondences define an isomorphism of
categories.
\end{proof}

\section{Quiver algebras and path coalgebras}\label{sectquiverpath}

 We examine the connection between the quiver algebra $K[\Gamma]$ and the path
coalgebra $K\Gamma$ associated to a quiver $\Gamma$. The algebra
$K[\Gamma]$ has identity if and only if $\Gamma$ has finitely many
vertices. However, $K[\Gamma]$ always has enough idempotents (the
set of all vertices). Thus by Section \ref{sectdualfinit} we can
consider the finite dual $K[\Gamma]^0$, which is a coalgebra with
counit. One has that $K[\Gamma]^0\supseteq K\Gamma$, i.e., the path
coalgebra can be embedded in the finite dual of the quiver algebra.
The embedding is given as follows: for each path $p\in\Gamma$,
denote by $\theta(p)\in K[\Gamma]^*$ the function
$\theta(p)(q)=\delta_{p,q}$. We have that $\theta(p)\in K[\Gamma]^0$
since if we denote by $S(p)$ the set of all subpaths of $p$, and by
$P$ the set of all paths in $\Gamma$, the span of $P\setminus S(p)$
is a finite codimensional ideal of $K[\Gamma]$ contained in $\Ker
\theta(p)$. It is easy to see that $\theta:K\Gamma\hookrightarrow
K[\Gamma]^0$ is an embedding of coalgebras.  In general,
$K[\Gamma]\hookrightarrow(K\Gamma)^*$ is surjective if and only if
the quiver $\Gamma$ is finite. Also, in general, $\theta$ is not
surjective. To see this, let $A$ be the quiver algebra of a loop
$\Gamma$, i.e., a quiver with one vertex and one arrow:
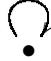
\vspace{.2cm}
$$\xymatrix{\bullet\ar@(ul,ur)[]}$$ so $A=K[X]$, the polynomial
algebra in one indeterminate. The finite dual of this algebra is
$\lim\limits_{\stackrel{\longrightarrow}{f{{\rm\,irreducible}};\,n\in\ZZ_{\geq
0}}}(K[X]/(f^n))^*=\bigoplus\limits_{f{\rm\,irreducible}}[\lim\limits_{\stackrel{\longrightarrow}{n\in
\ZZ_{\geq 0}}}(K[X]/(f^n))^*]$, while the path coalgebra is
precisely the divided power coalgebra, which can be written as
$\lim\limits_{\stackrel{\longrightarrow}{n\in \ZZ_{\geq
0}}}\left(K[X]/(X^n)\right)^*$. These two coalgebras are not
isomorphic, so the map $\theta$ above is not a surjection. Indeed,
$K\Gamma$ has just one grouplike element, the vertex of $\Gamma$,
while the grouplike elements of $A^0$, which are the algebra
morphisms from $A=K[X]$ to $K$, are in bijection to $K$.

The embedding of coalgebras $\theta:K\Gamma\hookrightarrow
K[\Gamma]^0$ also gives rise to a functor
$F^\theta:\,^{K\Gamma}\Mm\rightarrow \, ^{K[\Gamma]^0}\Mm$,
associating to a left $K\Gamma$-comodule the left
$K[\Gamma]^0$-comodule structure obtained by extension of coscalars
via $\theta$. We aim to provide a criterion for when the above map
$\theta$ is bijective, that is, when the path coalgebra is recovered
as the finite dual of the quiver algebra. Even though this is not
always the case, we show that it is possible to interpret the quiver
algebra as a certain kind of ``graded" finite dual. We will think of
$K\Gamma$ as embedded into $K[\Gamma]^0$ through $\theta$, and
sometimes write $K\Gamma$ instead of $\theta(K\Gamma)$.

Recall that in a quiver algebra $K[\Gamma]$, there is an important
class of ideals, those which have a basis of paths; equivalently,
the ideals generated by paths. Let us call such an ideal a monomial
ideal. When $I$ is a cofinite monomial right ideal, the quotient
$K[\Gamma]/I$ produces an interesting type of representation often
considered in the representation theory of quivers; we refer to
\cite{V} for the theory monomial algebras and representations. In
fact, such a representation also becomes a left $K\Gamma$-comodule,
i.e., it is in the ``image" of the functor $F^\theta$. To see this,
let $B$ be basis of paths for $I$ and let $E$ be the set of paths
not belonging to $I$; then $E$ is finite, and because $I$ is a right
ideal, one sees that if $p\in E$ and $p=qr$, then $q\in E$. This
shows that $KE$, the span of $E$, is a right $K\Gamma$-subcomodule
of $K\Gamma$, so it is a rational left $(K\Gamma)^*$-module (for
example, by \cite[Theorem 2.2.5]{DNR}). By \cite[Lemma 2.2.12]{DNR},
the right $(K\Gamma)^*$-module $(KE)^*$ is  rational, and so it has
a compatible left $K\Gamma$-comodule structure. Hence $(KE)^*$ is a
right $K[\Gamma]$-module via the algebra map
$K[\Gamma]\hookrightarrow (K\Gamma)^*$. Now, it is straightforward
to see that $K[\Gamma]/I\cong (KE)^*$ as right $K[\Gamma]$-modules,
and this proves the claim. Thus, $F^\theta(\,
^{K\Gamma}((KE)^*{}))=\, ^{K[\Gamma]^0}(K[\Gamma]/I)$, since every
finite dimensional right $K[\Gamma]$-module is a left
$K[\Gamma]^0$-comodule.

We can now state a characterization of the path coalgebra in terms
of the quiver algebra, as a certain type of finite dual.

\begin{proposition}\label{p.qalco}
The coalgebra $\theta(K\Gamma)$ consists of all $f\in K[\Gamma]^*$
such that $\ker(f)$ contains a two-sided cofinite monomial ideal.
\end{proposition}
\begin{proof}
Let $P$ be the set of paths in $\Gamma$. If $p$ is a path, and
$S(p)$ is the set of subpaths of $p$, then the cofinite dimensional
vector space with basis $P\setminus S(p)$ is an ideal, and it is
obviously contained in $\ker(\theta(p))$. Then clearly
$\ker(\theta(z))$ contains a cofinite monomial ideal for any $z\in
K\Gamma$.\\
Let now $f\in K[\Gamma]^*$ such that $\ker(f)$ contains the cofinite
monomial ideal $I$. Let $B$ be a basis of $I$ consisting of paths,
and let $E=P\setminus B$, which is finite, since $I$ is cofinite.
 Then if $q\in B$, we have that
$f(q)=0=\sum_{p\in E}f(p)\theta(p)(q)$, while if $q\in E$, we have
that $(\sum_{p\in E}f(p)\theta(p))(q)=f(q)$. Therefore $f=\sum_{p\in
E}f(p)\theta(p)\in \theta(K\Gamma)$.
\end{proof}

The following easy combinatorial condition will be the core of our
characterization.

\begin{proposition}\label{p.comb}
Let $\Gamma$ be a quiver. The following conditions are equivalent:\\
(i) $\Gamma$ has no oriented cycles and between any two vertices there are only finitely many arrows.\\
(ii) For any finite set of vertices $E\subset \Gamma$, there are
only finitely many paths passing only through vertices of $E$.
\end{proposition}

We recall that a representation of the quiver $\Gamma$ is a pair
$\mathcal{R}=((V_u)_{u\in \Gamma_0}, (f_a)_{a\in \Gamma_1})$
consisting of a family of vector spaces and a family of linear maps,
such that $f_a:V_u\rightarrow V_v$, where $u=s(a)$ and $v=t(a)$ for
any $a\in \Gamma_1$. A morphism of representations is a family of
linear maps (indexed by $\Gamma_0$) between the corresponding
$V_u$'s, which are compatible with the corresponding linear
morphisms in the two representations. The category $Rep\, \Gamma$ of
representations of $\Gamma$ is equivalent to the category
$u.\mathcal{M}_{K[\Gamma]}$ of unital right $K[\Gamma]$-modules. The
equivalence $H: u.\mathcal{M}_{K[\Gamma]}\rightarrow Rep\, \Gamma$
works as follows. To a unital right $K[\Gamma]$-module $V$ we
associate the representation $H(V)=((V_u)_{u\in \Gamma_0},
(f_a)_{a\in \Gamma_1})$, where $V_u=Vu$ for any $u\in \Gamma_0$, and
for an arrow $a$ from $u$ to $v$ we define $f_a:V_u\rightarrow V_v$,
$f_a(x)=xa$. An inverse equivalence functor associates to
representation $\mathcal{R}$ as above the space $\oplus_{u\in
\Gamma_0}V_u$ endowed with a right $K[\Gamma]$-action defined by
$xp=f_p(x)$ for $p=a_1\ldots a_n$ and $x\in V_u$ such that
$s(a_1)=u$. Here we denote $f_p=f_{a_n}\ldots f_{a_1}$. If
$s(a_1)\neq u$, the action is $xp=0$.

A representation $\mathcal{R}$ is locally finite if for any $u\in
\Gamma_0$ and any $x\in V_u$ the subspace $\langle f_p(x)| p \mbox{
path with }s(p)=u\rangle$ of $\oplus_{u\in \Gamma_0}V_u$ is finite
dimensional. Denote the subcategory of locally finite
representations by $LocFin Rep\, \Gamma$. The equivalence $H$
restricts to an equivalence $H_1:LocFin{\mathcal
M}_{K[\Gamma]}\rightarrow LocFin Rep\, \Gamma$.

Recall from \cite{chin} that a representation $\mathcal{R}$ is
locally nilpotent if for any $u\in \Gamma_0$ and any $x\in V_u$, the
set $\{ p|\, p \mbox{ path with }f_p(x)\neq 0\}$ is finite. This is
easily seen to be equivalent to each $x\in \oplus_{u\in
\Gamma_0}V_u$ being annihilated by a monomial ideal of finite
codimension. Denote by $LocNilp Rep\, \Gamma$ the category of
locally nilpotent representations, which is clearly a subcategory of
$LocFin Rep\, \Gamma$.

We have the following diagram

$$\xymatrix{
{}^{K\Gamma}\Mm \ar[r]^{F^\theta}\ar[d]^{\sim}_{H_2} &
{}^{K[\Gamma]^0}\Mm\ar[r]^(.38)G_(.38){\sim} &
LocFin \Mm_{K[\Gamma]} \ar[r]^(.55){I_1}\ar[d]_{H_1}^{\sim} & u.\Mm_{K[\Gamma]}\ar[d]_H^\sim \\
LocNilp Rep\,\Gamma \ar[rr]^{I_2} & & LocFin Rep\,\Gamma
\ar[r]^(.55){I_3} & Rep\,\Gamma }$$

Here $G$ is the equivalence of categories as in Proposition
\ref{propequivrepcorep} (the version for right modules), and the
$I_j$'s are inclusion functors. It is easy to see that the image (on
objects) of the functor $H_1GF^{\theta}$ lies in $LocNilp Rep\,
\Gamma$, so we actually have a functor $H_2:\,
^{K\Gamma}\mathcal{M}\rightarrow LocNilp Rep\, \Gamma$, and this is
just the equivalence  noticed in \cite[Proposition 6.1]{chin}. In
this way, at the level of representations, the functor $F^\theta$
can be regarded as a functor (embedding) from the locally nilpotent
quiver representations to the locally finite quiver representations.

We can now characterize precisely when the path coalgebra can be
recovered from the quiver algebra, that is, when the above mentioned
embedding $\theta$ is an isomorphism.

\begin{theorem}\label{t.pathalgco}
Let $\Gamma$ be a quiver. The following assertions are equivalent:\\
(i) $\Gamma$ has no oriented cycles and between every two vertices of $\Gamma$ there are only finitely many arrows.\\
(ii) $\theta(K\Gamma)=K[\Gamma]^0$.\\
(iii) Every cofinite ideal of $K[\Gamma]$ contains a cofinite monomial ideal.\\
(iv) The functor $F^\theta:\,^{K\Gamma}\Mm\rightarrow
\, ^{K[\Gamma]^0}\Mm$ is an equivalence.\\
(v) Every locally finite quiver representation of $\Gamma$ is
locally nilpotent.
\end{theorem}
\begin{proof}
The equivalence of $(ii)$ and $(iv)$ is a general coalgebra fact: if
$C\subseteq D$ is an inclusion of coalgebras, then the corestriction
of scalars $F:\,^C\Mm\rightarrow \, ^D\Mm$ is an equivalence if and
only if $C=D$. Indeed, if $F$ is an equivalence, pick an arbitrary
$x\in D$ and let $N=xD^*\in \,^D\Mm$ be the finite dimensional
$D$-subcomodule of $D$ generated by $x$. Then $N\simeq F(M)$, $M\in
\, ^C\Mm$, and considering the coalgebras  of
coefficients $C_N$ and $C_M$ of $N$ and $M$, we see that $C_N=C_M\subseteq C$ by the definition of $F$. Since $x\in C_N$, this ends the argument.\\
The equivalence of $(ii)$ and $(iii)$ follows immediately from Proposition \ref{p.qalco}.\\
$(iv)\Leftrightarrow (v)$ The previous remarks on $F^\theta$ (and
the diagram drawn there) show that $F^\theta$ is an equivalence
functor if and only if so is $I_2$. On the other hand, the inclusion
functor $I_2$ is an equivalence if and only if every locally finite
quiver representation of $\Gamma$ is
locally nilpotent.\\
$(i)\Rightarrow (iii)$
 Let $I$ be an ideal of $K[\Gamma]$ of finite codimension. By Lemma
 \ref{lemaidealcofinit} applied for the algebra $K[\Gamma]$ and the
 complete set of orthogonal idempotents $\Gamma_0$, we have that
  the set $S'=\{a \in
\Gamma_0| a\not\in I\}$ must be finite.
 Let $S=\{a \in
\Gamma_0| a\in I\}$. Note that any path $p$ starting or ending at a
vertex in $S$ belongs to $I$, since $p=s(p)p=pt(p)\in I$ if either
$s(p)\in I$ or $t(p)\in I$. Furthermore, this shows that if $p$
contains a vertex in $S$, then $p\in I$, since in that case $p=qr$
with $x=t(q)=s(r)\in S$. Denote the set of paths containing some
vertex in $S$ by $M$. Let $H$ be the vector space spanned by $M$ and
let $M'$ be the set of the rest of the paths in $\Gamma$. Obviously,
$M'$ consists of the paths whose all vertices belong to $S'$. Since
$S'$ is finite, we see that $M'$ is finite, by the conditions of
$(i)$ and Proposition \ref{p.comb}. Therefore $H$ has finite
codimension.
 Also, since $H$ is spanned by paths passing
through some vertex in $S$, we see that $H$ is an ideal. We conclude
that $I$ contains the cofinite monomial ideal $H$.\\
$(iii)\Rightarrow (i)$ We show first that there are no oriented
cycles in $\Gamma$. Assume $\Gamma$ has a cycle
$C:v_0\stackrel{x_0}{\longrightarrow}v_1\stackrel{x_1}{\longrightarrow}\dots
\longrightarrow v_{s-1}\stackrel{x_{s-1}}{\longrightarrow}v_s=v_0$,
and consider such a cycle that does not self-intersect. We can
consider the vertices $v_0,\dots,v_{s-1}$ modulo $s$. Denote by
$q_{n,k}$ the path starting at the vertex $v_n$ ($0\leq n \leq
s-1$), winding around the cycle $C$ and of length $k$. Denote again
by $P$ the set of all paths in $\Gamma$, and by $X=\{q_{n,k}|0\leq n
\leq s-1, k\geq 0\}$. Since the set $X$ is closed under subpaths, it
is easy to see that the vector space $H$ spanned by the set
$P\setminus X$ is an ideal of $K[\Gamma]$. Let $E$ be the subspace
spanned by $S=\{q_{n,ks+i}-q_{n,i}|0\leq n\leq s-1, i\geq 0, k\geq
1\}$, and let $I=E+H$. We have that
\begin{eqnarray*}
(q_{n,ks+i}-q_{n,i})q_{n+i,j} & = & q_{n,ks+i+j}-q_{n,i+j}\in S \\
(q_{n,ks+i}-q_{n,i})q_{m,j} & = & 0 \;\; \mbox{for}\;\; m\neq n+i\\
q_{m,j}(q_{n,ks+i}-q_{n,i}) & = & q_{m,ks+i+j}-q_{m,i+j}\in S \;\;
\mbox{if}\;\; m+j=n\\
q_{m,j}(q_{n,ks+i}-q_{n,i}) & = & 0 \;\; \mbox{if}\;\; m+j\neq n
\end{eqnarray*}
Here in the notation $q_{n,i}$ the first index is considered modulo
$s$, while the second index is a nonnegative integer. The above
equations show that if we multiply an element of $S$ to the left (or
right) by an element of $X$, we obtain either an element of $S$ or
0. Combined with the fact that $H$ is an ideal, this
shows that $I$ is an ideal.\\
It is clear that $I$ has finite codimension, since $S\cup \{
q_{n,i}|0\leq n\leq s-1,0\leq i\leq s-1\}$ spans $KC=\langle X
\rangle$. Indeed, if $0\leq n\leq s-1$ and $j$ is a non-negative
integer, write $j=ks+i$ with $k\geq 0$ and $0\leq i\leq s-1$, and we
have that
$q_{n,j}=q_{n,ks+i}=(q_{n,ks+i}-q_{n,i})+q_{n,i}$.\\
On the other hand, $I$ does not contain a cofinite monomial ideal.
Indeed, it is easy to see that an element of the form $q_{m,j}$
cannot be in $\langle S\cup (P\setminus X) \rangle=I$, so any
monomial ideal
contained in $I$ must have infinite codimension.\\
 Thus,
we have found a cofinite ideal $I$ which does not contain a cofinite
monomial ideal. This contradicts $(iii)$, and we conclude that $\Gamma$ cannot contain cycles.\\
We now show that in $\Gamma$ there are no pair of vertices with
infinitely many arrows between them. Assume such a situation exists
between two vertices $a,b$: $a\stackrel{x_n}{\longrightarrow}b$,
$n\in\ZZ_{\geq 0}$. We let $X=\{x_n|n\in\ZZ_{\geq 0}\}\cup \{a,b\}$,
$H$ be the span of $P\setminus X$, which is an ideal since $X$ is
closed under taking subpaths. Let $S=\{x_n-x_0|n\geq 1\}$ and $I$ be
the span of $S\cup (P\setminus X)$. As above, since $x_n-x_0$
multiplied by an element of $X$ gives either $x_n-x_0$ or 0, we have
that $I$ is an ideal. $I$ has finite codimension since $\{ a,b,x_0\}
\cup S\cup (P\setminus X)$ spans $K\Gamma$. Also, $I$ does not
contain a monomial ideal of finite codimension since no $x_n$ lies
in $I$. Thus we contradict $(iii)$. In conclusion there are finitely
many arrows between any two vertices, and this ends the proof.
\end{proof}

It is clear that a finite quiver $\Gamma$ (i.e. $\Gamma_0$ and
$\Gamma_1$ are finite) without oriented cycles satisfies condition
(i) in Theorem 3.3. In this case $K\Gamma$ is finite dimensional,
and we obviously have $K[\Gamma]=(K\Gamma)^*$ (i.e. the map $\theta$
is bijective) and also $K\Gamma=K[\Gamma]^0$. This can also be
thought as a trivial case of the above theorem.



We now present a few examples to further illustrate the above
theorem.

\begin{example}
Let ${}_\infty{\mathbb A}_{\infty}$ be the infinite line quiver
$$\dots \rightarrow\bullet\rightarrow \bullet\rightarrow \dots \rightarrow \bullet\rightarrow \dots$$
The quiver coalgebra $C=K{}_\infty{\mathbb A}_{\infty}$ of this
quiver is serial, that is, the injective indecomposable left and
right comodules are uniserial, i.e. they have a unique composition
series (see \cite{GTN}). For such a coalgebra, the finite
dimensional comodules are easily classified: they are all serial
(see \cite{GTN}). Moreover, the indecomposable finite dimensional
comodules, i.e. the uniserial ones, correspond to finite paths in
${}_\infty{\mathbb A}_\infty$. Note that this quiver satisfies the
conditions of Theorem \ref{t.pathalgco}, and so the locally
nilpotent representations of ${}_\infty{\mathbb A}_\infty$ (i.e. the
comodules over $K{}_\infty{\mathbb A}_\infty$) coincide with the
locally finite representations of the quiver algebra
$A=K[{}_\infty{\mathbb A}_\infty]$, and also, the finite dimensional
quiver representations of ${}_\infty{\mathbb A}_\infty$ are the
comodules over $K{}_\infty{\mathbb A}_\infty$. Moreover, the
coalgebra of
representative functions on $K[{}_\infty{\mathbb A}_\infty]$ is isomorphic to $K{}_\infty{\mathbb A}_\infty$.\\
Note that in general, it is not easy to describe arbitrary comodules
even for a serial coalgebra. By results in \cite{I1}, if an infinite
dimensional indecomposable injective comodule exists, then there are
comodules which do not decompose into indecomposable comodules (and,
in particular, are not indecomposable). Moreover, for the left
bounded infinite quiver ${\mathbb A}_\infty: \bullet\rightarrow
\bullet \rightarrow \dots \rightarrow \bullet \rightarrow \dots$, it
is shown in \cite{I1} that all left comodules over $K{\mathbb
A}_\infty$ are serial direct sums of indecomposable uniserial
comodules corresponding to finite paths, while in the category of
right comodules over $K{\mathbb A}_\infty$ there are objects which
do not decompose into direct sums of indecomposables.
\end{example}

\begin{example}
Let ${\mathbb C}_n$ be the following quiver of affine Dynkin type $\widetilde{A_n}$:\\
$
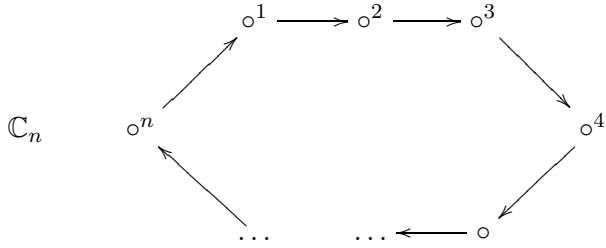
\xymatrix{& &\circ^1\ar[r] & \circ^2\ar[r] & \circ^3\ar[dr] &  \\
\CC_n  & \circ^n\ar[ur] & & & & \circ^4\ar[dl]\\
& & \dots\ar[ul]& \dots & \circ\ar[l] &}$\\
The path coalgebra $K{\mathbb C}_n$ is again serial, and the finite
dimensional (left and right) comodules are all direct sum of
uniserial objects (corresponding to finite paths). These correspond
to finite dimensional locally nilpotent representations. This quiver
does not satisfy the hypothesis of Theorem \ref{t.pathalgco}. We
give an example of a quiver representation which is locally finite
(even finite dimensional) but not locally nilpotent. Let $x_1,\dots,
x_n$ denote the arrows of the quiver, with $a_i=s(x_i)$. Let
$M=K[{\mathbb C}_n]/I$ where $I$ is the (two sided) ideal generated
by elements $p-1$, where $p$ is a path of length $n$ and with
$1=a_1+\dots +a_n$ ($M$ is actually an algebra). One can easily see
that $M$ is spanned as a vector space by paths of length less than
$n$. A not too difficult computation shows that $I$ does not contain
any monomial ideal of finite codimension, and so $M$ as a
representation of $K[{\mathbb{C}}_n]$ is not locally nilpotent, but
it is finite dimensional. We again note that the infinite
dimensional comodules over the coalgebra $K{\mathbb C}_n$ are hard
to
understand, as there are both left and right comodules which are not direct sum of indecomposable comodules.\\
An easy particular example of this can be obtained for $n=1$; in
this case, $K[{\mathbb C}_n]\cong K[X]$ - the polynomial algebra. As
noted before, the finite dual of this algebra is not the path
coalgebra of ${\mathbb C}_1$. Also, the representation $K[X]/(X-1)$
is not locally nilpotent.
\end{example}

Let $\psi:K[\Gamma]\rightarrow (K\Gamma )^*$ be the linear map
defined by $\psi(p)(q)=\delta_{p,q}$ for any paths $p$ and $q$. In
fact $\psi$ is just $\theta$ as a linear map, but we denote it
differently since we regard it now as a morphism in the category of
algebras not necessarily with identity. Indeed, it is easy to check
that $\psi$ is multiplicative. Thus the quiver algebra embeds in the
dual of the path coalgebra. Our aim is to show that in certain
situations $K[\Gamma]$ can be recovered from $(K\Gamma )^*$ as the
rational part. Obviously, this is the case when $K[\Gamma]$ is
finite dimensional, which will also be seen as a consequence of the
next result, which characterizes completely these situations. We
recall that if $C$ is a coalgebra, the rational part of the
 left $C^*$-module  $C^*$, consisting of all elements $f\in
 C^*$ such that there exist finite families $(c_i)_i$ in $C$ and
 $(f_i)_i$ in $C^*$ with $c^*f=\sum_ic^*(c_i)f_i$ for any $c^*\in
 C^*$, is denoted by $C^{*rat}_l$. This is the largest
 $C^*$-submodule which is rational, i.e., whose $C^*$-module
 structure comes from a right $C$-comodule structure. Similarly,
 $C^{*rat}_r$ denotes the rational part of the right $C^*$-module
 $C^*$. A coalgebra $C$ is called right (respectively left)
 semiperfect if the category of right (respectively left)
 $C$-comodules has enough projectives. This is equivalent to the
 fact that $C^{*rat}_l$ (respectively $C^{*rat}_r$) is dense in
 $C^*$ in the finite topology, see \cite[Section 3.2]{DNR}.

\begin{theorem} \label{thpathsemiperfect}
The following are equivalent.\\
(i) $Im(\psi)=(K\Gamma)^{*rat}_l$.\\
(ii) $Im(\psi)=(K\Gamma)^{*rat}_r$.\\
(iii) For any vertex $v$ of $\Gamma$ there are finitely many
paths starting at $v$ and finitely many paths ending at $v$.\\
(iv) The path coalgebra $K\Gamma$ is left and right semiperfect.
\end{theorem}
\begin{proof}
$(iii)\Rightarrow (i)$ Let $p$ be a path. We show that
$p^*=\psi(p)\in Im(\psi)$. If $c^*\in (K\Gamma)^*$ and $q$ is a
path, we have that
$$(c^*p^*)(q)=\sum_{rs=q}c^*(r)p^*(s)=\left\{
\begin{array}{l}
c^*(r),\mbox{ if }q=rp\;\mbox{for some path }r,\\
0, \mbox{ if }q\mbox{ does not end with }p.
\end{array}\right.$$
Let $q_1=r_1p,\ldots,q_n=r_np$ be all the paths ending with $p$. By
the formula above, $(c^*p^*)(q_i)=c^*(r_i)$ for any $1\leq i\leq n$,
and $(c^*p^*)(q)=0$ for any path $q\neq q_1,\ldots,q_n$. This shows
that $c^*p^*=\sum_{1\leq i\leq n}c^*(r_i)q^*_i$, thus $p^*\in
(K\Gamma)^{*rat}_l$, and we have that $Im(\psi)\subseteq
(K\Gamma)^{*rat}_l$.

Now let $c^*\in (K\Gamma)^{*rat}_l$, so there exist $(c_i)_{1\leq
i\leq n}$ in $K\Gamma$ and $(c^*_i)_{1\leq i\leq n}$ in
$(K\Gamma)^*$ such that $d^*c^*=\sum_{1\leq i\leq n}d^*(c_i)c^*_i$
for any $d^*\in (K\Gamma)^*$. Let $p_1,\ldots,p_m$ be all the paths
that appear with non-zero coefficients in some of the $c_i$'s
(represented as a linear combination of paths). Then for any $p\neq
p_1,\ldots,p_m$ we have that $p^*(c_i)=0$, so then $p^*c^*=0$. Let
$v$ be a vertex such that no one of $p_1,\ldots,p_m$ passes through
$v$. Then for any path $p$ starting at $v$ we have that
$0=(v^*c^*)(p)=v^*(v)c^*(p)=c^*(p)$. Therefore $c^*$ may be non-zero
on a path $p$ only if $s(p)\in \{p_1,\ldots,p_m\}$. By condition
$(iii)$, there are only finitely many such paths $p$, denote them by
$q_1,\ldots,q_e$. Then $c^*=\sum_{1\leq i\leq e}c^*(q_i)q_i^*\in
Im(\psi)$, and we also have that $(K\Gamma)^{*rat}_l\subseteq
Im(\psi)$.\\
$(i)\Rightarrow (iii)$ Let $v$ be a vertex. Then $v^*=\psi(v)\in
(K\Gamma)^{*rat}_l$, so there exist finite families $(c_i)\subseteq
K\Gamma$ and $(c^*_i)_i\subseteq (K\Gamma)^*$ such that
$c^*v^*=\sum_ic^*(c_i)c^*_i$ for any $c^*\in (K\Gamma)^*$. Then for
any path $q$,
\begin{equation} \label{eqvstar}
\sum_ic^*(c_i)c^*_i(q)=(c^*v^*)(q)=\left\{
\begin{array}{l}
c^*(q),\mbox{ if }q\;\mbox{ends at }v,\\
0, \mbox{ otherwise. }
\end{array}\right.
\end{equation}
If there exist infinitely many paths ending at $v$, we can find one
such path $q$ which does not appear in the representation of any
$c_i$ as a linear combination of paths. Then there exists $c^*\in
(K\Gamma)^*$ with $c^*(q)\neq 0$ and $c^*(c_i)=0$ for any $i$, in
contradiction with (\ref{eqvstar}). Thus only finitely many paths
can end at $v$. In particular $\Gamma$ does not have cycles.

On the other hand, if we assume that there are infinitely many paths
$p_1,p_2,\ldots$ starting at $v$, let $c^*\in (K\Gamma)^*$ which is
1 on each $p_i$ and 0 on any other path. Clearly $c^*\notin
Im(\psi)$. We show that $c^*\in (K\Gamma)^{*rat}_l$, and the
obtained contradiction shows that only finitely many paths start at
$v$. Indeed, we have that
\begin{equation}\label{eqcstar}
(d^*c^*)(q)=\left\{
\begin{array}{l}
d^*(r),\mbox{ if }q=rp_i\;\mbox{for some }i\geq 1\;\mbox{and some path }r\\
0, \mbox{ otherwise }
\end{array}\right.
\end{equation}
Let $r_1,\ldots,r_m$ be all the paths ending at $v$ (they are
finitely many as we proved above). For each $1\leq j\leq m$ we
consider the element $c^*_j\in (K\Gamma)^*$ which is 1 on every path
of the form $r_jp_i$, and 0 on any other path. Using (\ref{eqcstar})
and the fact that $r_jp_i\neq r_{j'}p_{i'}$ for $(i,j)\neq (i',j')$
(this follows because $r_j$, $r_{j'}$ end at $v$ and $p_i, p_{i'}$
start at $v$, and there are no cycles containing $v$), we see that
$d^*c^*=\sum_{1\leq j\leq m}d^*(r_j)c^*_j$, and
this will guarantee that $c^*$ is a rational element.\\
$(ii)\Leftrightarrow (iii)$ is similar to $(i)\Leftrightarrow (iii)$.\\
$(iii)\Leftrightarrow (iv)$ follows from \cite[Corollary 6.3]{chin}.
\end{proof}

\section{Incidence coalgebras and incidence
algebras}\label{sectincidence}

In this section we parallel the results in Section
\ref{sectquiverpath} in the framework of incidence (co)algebras. Let
$(X,\leq)$ be a partially ordered set which is locally finite, i.e.,
the set $\{z|x\leq z\leq y\}$ is finite for any $x\leq y$ in $X$.
The incidence coalgebra of $X$, denoted by $KX$, is the vector space
with basis $\{ e_{x,y}|x,y\in X, x\leq y\}$, and comultiplication
and counit  defined by $\Delta(e_{x,y})=\sum_{x\leq z\leq
y}e_{x,z}\otimes e_{z,y}$, $\epsilon (e_{x,y})=\delta_{x,y}$ for any
$x,y\in X$ with $x\leq y$. For such a $X$, we can consider the
quiver $\Gamma$ with vertices the elements of $X$, and such that
there is an arrow from $x$ to $y$ if and only if $x<y$ and there is
no element $z$ with $x<z<y$. It was proved in \cite{din} that the
linear map $\phi:KX\rightarrow K\Gamma$, defined by
$$\phi(e_{x,y})=\sum_{p\; {\rm path}\atop s(p)=x,t(p)=y}p$$ for any
$x,y\in X, x\leq y$, is an injective coalgebra morphism. We note
that this map is surjective if and only if in $\Gamma$ there is at
most one path between any to vertices $x,y\in X$. To see this, let
$P(x,y)$ denote the set of paths from $x$ to $y$. Note that the
incidence coalgebra $KX$ is then $KX=\bigoplus_{x,y\in X}\langle
P(x,y)\rangle$, and that $\phi(\langle e_{x,y} \rangle)\subset
P(x,y)$ for $x\leq y$. Thus, $\phi$ is surjective if and only if
$\dim(P(x,y))=1$ for all $x\leq y$, which is equivalent to the above
stated condition. In fact, this is also a consequence of the
following more general fact.

\begin{proposition}\label{p.IPcoalg}
A coalgebra $C$ is both an incidence coalgebra and a path coalgebra
if and only if it is the path coalgebra of a quiver $\Gamma$ for
which there is at most one path between any two vertices.
\end{proposition}
\begin{proof}
If the condition is satisfied for a quiver $\Gamma$, we can
introduce an obvious order on the set $X$ of vertices of $\Gamma$
setting $x\leq y$ if and only if there is a path from $x$ to $y$. It
is easy to check that this is an ordering, and so the above map
$\phi: KX\rightarrow K\Gamma$ is bijective. Conversely, let $C\cong
KX\cong K\Gamma$ for a locally finite partially ordered set $X$ and
a quiver $\Gamma$. We note that the simple subcoalgebras (and simple
left subcomodules, simple right subcomodules) of $C$ are precisely
the spaces $Kx$ for $x\in X$ and $Kv$ for $v$ vertex in $\Gamma$,
and $X$, respectively $\Gamma$ correspond to the group-like elements
of $C$. Thus, $X$ must be the set of vertices of $\Gamma$.
Furthermore, we note that in either an incidence coalgebra or a path
coalgebra, the injective hull of a simple left comodule $Kx$ is
uniquely determined as follows (note that in general, given an
injective module $M$ and a submodule $N$ of $M$, there is an
injective hull of $N$ contained in $M$ but it is not necessarily
uniquely determined). For incidence/path coalgebras, the right
(left) injective hull $E_r(Kx)$ of $Kx$ (respectively, $E_l(Kx)$) of
the right (respectively, left) comodule $Kx$ is the span of all
segments/paths starting (respectively, ending) at $x$ (see the proof
of \cite[Proposition 2.5]{simson} for incidence coalgebras and
\cite[Corollary 6.2(b)]{chin} for path coalgebras). Then, for $x\leq
y$, from the incidence coalgebra results we get $E_r(Kx)\cap
E_l(Ky)=\langle e_{x,y} \rangle$ and from the path coalgebra we get
$E_r(Kx)\cap E_l(Ky)=\langle P(x,y)\rangle$. This shows that
$\langle P(x,y)\rangle$ is one dimensional, and the proof is
finished.
\end{proof}


Apart from the incidence coalgebra $KX$, there is another associated
algebraic object with a combinatorial relevance. This is the
incidence algebra $IA(X)$, which is the space of all functions
$f:\{(x,y)|x,y\in X, x\leq y\}\rightarrow K$ (functions on the set
of closed intervals of $X$),
with multiplication given by convolution: $(fg)(x,y)=\sum_{x\leq
z\leq y}f(x,z)g(z,y)$ for any $f,g\in IA(X)$ and any $x,y\in X,
x\leq y$. See \cite{so} for details on the combinatorial relevance
of the incidence algebra. It is clear that $IA(X)$ is isomorphic to
the dual algebra of $KX$, if we identify a map $f\in IA(X)$ with the
element of $(KX)^*$ which takes $e_{x,y}$ to $f(x,y)$ for any $x
\leq y$. For simplicity, we will identify $IA(X)$ with $(KX)^*$.

Comparing to path coalgebras and quiver algebras, the situation is
different, since the incidence algebra always has identity. However,
we can consider the subspace $FIA(X)$ of $IA(X)$ spanned by all the
elements $E_{x,y}$ with $x\leq y$, where
$E_{x,y}(e_{u,v})=\delta_{x,u}\delta_{y,v}$ for any $u\leq v$.
Equivalently, $FIA(X)$ consists of all functions on $\{(x,y)|x,y\in
X, x\leq y\}$ that have finite support. Then $FIA(X)$ is a subalgebra
of $IA(X)$ which does not have an identity when $X$ is infinite, but
it has enough idempotents, the set of all $E_{x,x}$. The algebra
$FIA(X)$ plays the role of the quiver algebra in this new framework.

The subspace $FIA(X)$ is dense in $IA(X)$ in the finite topology,
since it is easy to see that $FIA(X)^\perp =0$ (see \cite[Corollary
1.2.9]{DNR}). We have a coalgebra morphism $\theta:KX\rightarrow
FIA(X)^0$, defined by $\theta(c)(c^*)=c^*(c)$ for any $c\in KX$ and
any $c^*\in FIA(X)$. We note that $\theta(c)$ indeed lies in
$FIA(X)^0$, since $\Ker (\theta (c))=\langle c\rangle ^\perp \cap
FIA(X)\supseteq D^\perp \cap FIA(X)$, where $D$ is the (finite
dimensional) subcoalgebra generated by $c$ in $KX$. Then $D^\perp$
is an ideal of $IA(X)$ of finite codimension, and then $D^\perp \cap
FIA(X)$ is an ideal of $FIA(X)$ of finite codimension. Since
$FIA(X)$ is dense in $IA(X)$, $\theta$ is injective. The next result
shows that we can recover the incidence coalgebra $KX$ as the finite
dual of the algebra with enough idempotents $FIA(X)$. The result
parallels Theorem \ref{t.pathalgco}; note that the conditions
analogous to the ones in $(i)$ in Theorem \ref{t.pathalgco} are
always satisfied in incidence algebras.

\begin{theorem}\label{t.4.2}
For any locally finite partially ordered set $X$, the map
$\theta:KX\rightarrow FIA(X)^0$ is an isomorphism of coalgebras.
\end{theorem}
\begin{proof}
It is enough to show that $\theta$ is surjective. Let $F\in
FIA(X)^0$, so $F:FIA(X)\rightarrow K$ and $Ker(F)$ contains an ideal
$I$ of $FIA(X)$ of finite codimension. Then the set $X_0=\{ x\in
X|E_{x,x}\notin I\}$ is finite by Lemma \ref{lemaidealcofinit}.

If $x\in X\setminus X_0$, then $E_{x,y}=E_{x,x}E_{x,y}\in I$ for any
$x\leq y$. Similarly $E_{x,y}\in I$ for any $y\in X\setminus X_0$
and  $x\leq y$. Thus in order to have $E_{x,y}\notin I$, both $x$
and $y$ must lie in $X_0$. This shows that only finitely many
$E_{x,y}$'s lie outside $I$.
 Let $\mathcal F$ be the set of all pairs $(x,y)$
such that $E_{x,y}\notin I$. Then we have that $F=\sum_{(x,y)\in
{\mathcal F}}F(E_{x,y})\theta(e_{x,y})$. Indeed, when evaluated at
$E_{u,v}$, both sides are 0 if $(u,v)\notin {\mathcal F}$, or
$F(E_{u,v})$ if $(u,v)\in {\mathcal F}$. Thus $F\in Im(\theta)$.
\end{proof}

The next result and its proof parallel Theorem
\ref{thpathsemiperfect}.

\begin{theorem} \label{thincidencesemiperfect}
Let $C=KX$. The following assertions are equivalent.\\
(i) $FIA(X)=C^{*rat}_l$.\\
(ii) $FIA(X)=C^{*rat}_r$.\\
(iii) For any $x\in X$ there are finitely many elements $u\in X$
such that $u\leq x$, and finitely many elements $y\in X$ such that
$x\leq y$.\\
(iv) $KX$ is a left and right semiperfect coalgebra.
\end{theorem}
\begin{proof}
$(i)\Rightarrow (iii)$ Since $E_{x,x}\in C^{*rat}_l$, there exist
finite families $(c_i)_i$ in $C$ and $(c^*_i)_i$ in $C^*$ such that
$c^*E_{x,x}=\sum_ic^*(c_i)c^*_i$ for any $c^*\in C^*$. If there are
infinitely many elements $u$ in $X$ such that $u\leq x$, then we can
choose such an element $u_0$ for which $e_{u_0,x}$ does not show up
in the representation of any $c_i$ (as a linear
combination of the standard basis of $C$). 
Since $E_{u_0,x}(e_{p,q})=\delta_{u_0,p}\delta_{x,q}$, we get
$E_{u_0,x}(c_i)=0$ for any $i$, so $\sum_iE_{u_0,x}(c_i)c^*_i=0$,
while $(E_{u_0,x}E_{x,x})(e_{u_0,x})=1$, a contradiction.

Assume now that for some $x\in X$ the set of all elements $y$ with
$x\leq y$, say $(y_i)_i$, is infinite. Let $c^*\in C^*$ which is 1
on each $e_{x,y_i}$ and 0 on any other $e_{p,q}$. Then it is easy to
see that
$$
(d^*c^*)(e_{u,v})=\left\{
\begin{array}{l}
d^*(e_{u,x}),\mbox{ if } u\leq x\leq v, \mbox{ and } v\in\{y_i|i\}\\
0, \mbox{ otherwise }
\end{array}\right.
$$
Let $(u_j)_j$ be the family of all elements $u$ with $u\leq x$. As
we proved above, this family is finite. For each $j$, let $c^*_j\in
C^*$ equal 1 on every $e_{u_j,y_i}$, and 0 on any other $e_{p,q}$.
We have that $d^*c^*=\sum_jd^*(e_{u_j,x})c^*_j$ for any $d^*\in
C^*$. Indeed, using the formula above we see that both sides equal
$d^*(e_{u_{j_0},x})$ when evaluated at some $e_{u_{j_0},y_i}$, and 0
when evaluated at any other $e_{p,q}$.

Therefore $c^*\in C^{*rat}_l$, but obviously $c^*\notin FIA(X)$,
since it is non-zero on infinitely many $e_{p,q}$'s. \\
$(iii)\Rightarrow (i)$ Choose some $x,y$ with $x\leq y$. Then for
any $c^*\in C^*$ we have that
$$
(c^*E_{x,y})(e_{u,v})=\left\{
\begin{array}{l}
c^*(e_{u,x}),\mbox{ if } u\leq x \leq y=v\\
0, \mbox{ otherwise }
\end{array}\right.
$$
This shows that if $(u_j)_j$ is the finite family of all elements
$u$ with $u\leq x$, then $c^*E_{x,y}=\sum_jc^*(e_{u_j,x})E_{u_j,y}$,
so $E_{x,y}$ lies in $C^{*rat}_l$.

Now let $c^*\in C^{*rat}_l$, so $d^*c^*=\sum_id^*(c_i)c^*_i$ for
some finite families $(c_i)_i$ in $C$ and $(c^*_i)_i$ in $C^*$. If
$x\in X$ such that $e_{x,x}$ does not appear in any $c_i$ (with
non-zero coefficient), then $E_{x,x}c^*=0$. In particular
$0=(E_{x,x}c^*)(e_{x,y})=c^*(e_{x,y})$ for any $x\leq y$. Since only
finitely many $e_{u,u}$ appear in the representations of the
$c_i$'s, and for any such $u$ there are finitely many $v$ with
$u\leq v$, we obtain that $c^*(e_{u,v})$ is non-zero for only
finitely many $e_{u,v}$. So $c^*$ lies in the span of all
$E_{x,y}$'s,
which is $FIA(X)$. \\
$(ii)\Leftrightarrow (iii)$ is similar.\\
$(iii)\Leftrightarrow (iv)$ follows from \cite[Lemma 5.1]{simson}.
\end{proof}

\section{Coreflexivity for path subcoalgebras and subcoalgebras of incidence
coalgebras}\label{s5}

We recall from \cite{Rad,T1} that a coalgebra $C$ is called
coreflexive if any finite dimensional left (or equivalently, any
finite dimensional right) $C^*$-module is rational. This is also
equivalent to asking that the natural embedding of $C$ into the
finite dual of $C^*$, $C\rightarrow (C^*)^0$ is surjective (so an
isomorphism), or that any left (equivalently, any right) cofinite
ideal is closed in the finite topology. See \cite{R,Rad,T1,T2} for
further equivalent characterizations.

Given the definition of coreflexivity and the characterizations of
the previous section, it is natural to ask what is the connection
between the situation when the path coalgebra can be recovered as
the finite dual of the quiver algebra, and the coreflexivity of the
path coalgebra. We note that these two are closely related. We have
an embedding $\iota:K\Gamma\hookrightarrow (K\Gamma)^*{}^0$; at the
same time, we note that the embedding of algebras (without identity)
$\psi:K[\Gamma]\hookrightarrow (K\Gamma)^*$ which is dense in the
finite topology of $(K\Gamma)^*$, produces a comultiplicative
morphism  $\varphi:(K\Gamma)^*{}^0\rightarrow K[\Gamma]^0$. Note
that $\varphi$ is not necessarily a morphism of coalgebras, since it
may not respect the counits. It is easy to see that these canonical
morphisms are compatible with $\theta$, i.e., they satisfy
$\theta=\varphi\circ\iota$:
$$\xymatrix{
K\Gamma \ar[r]_{\hookrightarrow}^\iota \ar[dr]_{\theta} & (K\Gamma)^*{}^0 \ar[d]^\varphi\\
 & K[\Gamma]^0
}$$ It is then natural to ask what is the connection between the
bijectivity of $\theta$, and coreflexivity of $K\Gamma$, i.e.,
bijectivity of $\iota$. In fact, we notice that if $C$ is
coreflexive (equivalently, $\iota$ is surjective), then $\varphi$ is
necessarily injective.

The following two examples will show that, in fact, $C$ can be
coreflexive and $\theta$ not an isomorphism, and also that $\theta$
can be an isomorphism without $C$ being coreflexive.

\begin{example}
Consider the path coalgebra of the following quiver $\Gamma$:
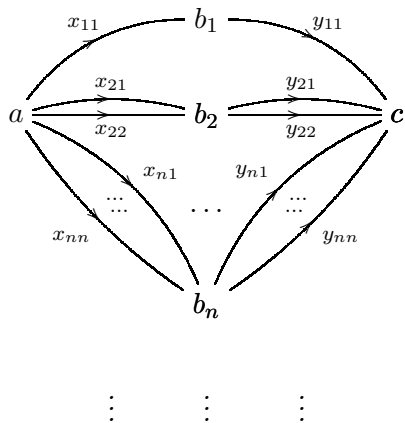
$$\xygraph{ 
!~:{@{-}|@{>}}
a(:@/^1pc/[urr]{b_1}^{x_{11}}:@/^1pc/[drr]{c}^{y_{11}}"a",
:@/^.5pc/[rr]{b_2}^{x_{21}}:@/^.5pc/[rr]{c}^{y_{21}}"a",
:@/^0pc/[rr]{b_2}_{x_{22}}:@/^0pc/[rr]{c}_{y_{22}}[dll]{\dots}"a",
:@/^1pc/[ddrr]{b_n}^{x_{n1}}_\dots:@/^1pc/[uurr]{c}^{y_{n1}}_\dots
"a",:@/^-.5pc/[ddrr]{b_n}_{x_{nn}}^\dots:@/^-.5pc/[uurr]{c}_{y_{nn}}^\dots
[dddlll]{\vdots}[r]{\vdots}[r]{\vdots}
)}$$ Here there are $n$ arrows from vertex $a$ to vertex $b_n$ and
$n$ arrows from $b_n$ to $c$ for each positive integer $n$. We note
that the one-dimensional vector space $I$ spanned by $a-c$ is a
coideal, since $a-c$ is an $(a,c)$-skew-primitive element. It is not
difficult to observe that the quotient coalgebra $C/I$ is isomorphic
to the coalgebra from \cite[Example 3.4]{R}, and so $C/I$ is not
coreflexive as showed in \cite{R}. By \cite[3.1.4]{HR}, we know that
if $I$ is a finite dimensional coideal of a coalgebra $C$ then $C$
is coreflexive if and only if $C/I$ is coreflexive. Therefore, $C$
is not coreflexive. However, it is obvious that $C$ satisfies the
quiver conditions of Theorem \ref{t.pathalgco}, and therefore,
$K\Gamma=K[\Gamma]^0$.
\end{example}

Hence, a path coalgebra of a quiver with no cycles and finitely many
arrows between any two vertices is not necessarily coreflexive.
Conversely, we note that in a coreflexive path coalgebra there are
only finitely many arrows between any two vertices. This is true
since a coreflexive coalgebra is locally finite by \cite[3.2.4]{HR},
which means that the wedge $X\wedge Y=\Delta^{-1}(X\otimes
C+C\otimes Y)$ of any two finite dimensional vector subspaces $X,Y$
of $C$ is finite dimensional (one applies this for $X=Ka$ and
$Y=Kb$). However, if  a path coalgebra $K\Gamma$ is coreflexive,
$\Gamma$ may contain cycles: consider the path coalgebra $C$ of a
loop (a graph with one vertex and one arrow);
$C$ is then the divided power coalgebra, $C^*=K[[X]]$, the ring of
formal power series, and its only ideals are $(X^n)$, which are
closed in the finite topology of $C^*$. Thus, every finite
dimensional $C^*$-module is rational and $C$ is coreflexive.

We will prove coreflexivity of an interesting class of path
coalgebras, whose quiver satisfy a slightly stronger condition than
that required by Theorem \ref{t.pathalgco} (so in particular, they
will satisfy $K\Gamma=K[\Gamma]^0$). We first prove a general
coreflexivity criterion.

\begin{theorem} \label{thgeneralcoreflexive}
Let $C$ be a coalgebra with the property that for any finite
dimensional subcoalgebra $V$ there exists a finite dimensional
subcoalgebra $W$ such that $V\subseteq W$ and
$W^{\perp}W^{\perp}=W^\perp$. Then $C$ is coreflexive if and only if
its coradical $C_0$ is coreflexive.
\end{theorem}
\begin{proof}
If $C$ is coreflexive, then so is $C_0$, since a subcoalgebra of a
coreflexive coalgebra is  coreflexive (see \cite[Proposition
3.1.4]{HR}). Conversely, let $C_0$ be coreflexive. We prove that any
finite dimensional left $C^*$-module $M$ is rational, by induction
on the length $l(M)$ of $M$. If $l(M)=1$, i.e., $M$ is simple, then
$M$ is also a left $C^*/J(C^*)$-module, where $J(C^*)$ is the
Jacobson radical of $C^*$. Since $C^*/J(C^*)\simeq C_0^*$ and $C_0$
is coreflexive, we have that $M$ is a rational $C_0^*$-module, so
then it is a rational
$C^*$-module, too. \\
Assume now that the statement is true for length $<n$, where $n>1$,
and let $M$ be a left $C^*$-module of length $n$. Let $M'$ be a
simple submodule of $M$, and consider the associated exact sequence
$$0\rightarrow M'\rightarrow M \rightarrow M''\rightarrow 0$$
By the induction hypothesis $M'$ and $M''$ are rational. By
\cite[Theorem 2.2.14]{DNR} we have that $ann_{C^*}(M')$ and
$ann_{C^*}(M'')$ are finite codimensional closed two-sided ideals in
$C^*$. Using \cite[Corollary 1.2.8 and Proposition 1.5.23]{DNR},
 $ann_{C^*}(M')=U_1^\perp$ and
$ann_{C^*}(M'')=U_2^\perp$ for some finite dimensional subcoalgebras
of $C$. Using the hypothesis for $V=U_1+U_2$, there is a finite
dimensional subcoalgebra $W$ of $C$ such that $U_1\subseteq W$,
$U_2\subseteq W$ and $W^{\perp}W^{\perp}=W^\perp$. Then by
\cite[Proposition 1.5.23]{DNR}, $W^\perp=W^{\perp}W^{\perp}\subseteq
U_1^{\perp}U_2^{\perp}=ann_{C^*}(M')ann_{C^*}(M'')\subseteq
ann_{C^*}(M)$ is a two-sided closed ideal of $C^*$, of finite
codimension. Therefore, $M$ is a rational $C^*$-module by using
again \cite[Theorem 2.2.14]{DNR}.
\end{proof}

\begin{proposition} \label{lemaconditii}
Let $C$ be the path coalgebra $K\Gamma$, where $\Gamma$ is a quiver
such that there are finitely many paths between any two vertices.
Then for any finite dimensional subcoalgebra $V$ of $C$ there exists
a finite dimensional subcoalgebra $W$ such that $V\subseteq W$ and
$W^{\perp}W^{\perp}=W^\perp$. As a consequence, $C$ is coreflexive
if and only if the coradical $C_0$ (which is the grouplike coalgebra
over  the set of vertices of $\Gamma$) is coreflexive.
\end{proposition}
\begin{proof}
Let $V$ be a finite dimensional subcoalgebra of $C=K\Gamma$. An
element $c\in V$ is of the form
$c=\sum\limits_{i=1}^{n}\alpha_ip_i,\alpha_i\neq 0$, a linear
combination of paths $p_1,\ldots ,p_n$. Consider the set of all
vertices at least one of these paths passes through, and let $S_0$
be the union of all these sets of vertices when $c$ runs through the
elements of $V$. Since $V$ is finite dimensional, we have that $S_0$
is finite (in fact, one can see that $S_0$ consists of all vertices
in $\Gamma$ which belong to $V$, so that $KS_0$ is the socle of
$V$). Let $P$ be the set of all paths $p$ such that $s(p),t(p)\in
S_0$. We consider the set $S$ of all vertices at least one path of
$P$ passes through. It is clear that $ P$ is finite, and then so is
$S$. We note that if $v_1,v_2\in S$ and $p$ is a path from $v_1$ to
$v_2$, then any vertex on $p$ lies in $S$. Indeed, $v_1$ is on a
path from $u_1$ to $u_1'$ (vertices in $S_0$), and let $p_1$ be its
subpath from $u_1$ to $v_1$. Similarly, $v_2$ is on a path from
$u_2$ to $u_2'$ (in $S_0$), and let $p_2$ be the subpath from $v_2$
to $u_2'$. Then $p_1pp_2\in{P}$, so any vertex of $p$ is in $S$. Let
$W$ be the subspace spanned by all paths starting and ending at
vertices in $S$. It is clear that any subpath of a path in $W$ is
also in $W$, so then $W$ is a finite dimensional subcoalgebra
containing $V$ (since $S_0$ is contained in $S$).\\
 We show that $W^{\perp}W^{\perp}=W^\perp$. For
this, given $\eta \in W^{\perp}$, we construct $f_1,f_2,g_1,g_2\in
W^{\perp}$ such that $\eta=f_1g_1+f_2g_2$. We define $f_i(p)$ and
$g_i(p)$, $i=1,2$, on all paths $p$  by induction on the length of
$p$. For paths $p$ of length zero, i.e., if $p$ is a vertex $v$, we
define $f_i(v)=g_i(v)=0$, $i=1,2$, for any $v\in S$, while for
$v\notin S$, we set $f_1(v)=g_2(v)=1$, and $f_2(v)$ and $g_2(v)$ are
such that $g_1(v)+f_2(v)=\eta (v)$. Then clearly
$\eta=f_1g_1+f_2g_2$ on paths of length zero. For the induction
step, assume that we have defined $f_i$ and $g_i$, $i=1,2$, on all
paths of length $<l$, and that $\eta=f_1g_1+f_2g_2$ on any such
path. Let now $p$ be a path of length $l$, starting at $u$ and
ending at $v$. If $u,v\in S$, then we define $f_i(p)=g_i(p)=0$,
$i=1,2$, and clearly $\eta (p)=\sum_{i=1,2}\sum_{qr=p}f_i(q)g_i(r)$,
since both sides are zero. If either $u\notin S$ or $v\notin S$, we
need the following equality to hold.
\begin{equation} \label{eq1}
f_1(u)g_1(p)+f_1(p)g_1(v)+f_2(u)g_2(p)+f_2(p)g_2(v)=\eta(p)-\sum_{i=1,2}\sum_{qr=p\atop
{q\neq p,r\neq p}}f_i(q)g_i(r)
\end{equation}
We note that the terms of the right-hand side of the equality
(\ref{eq1}) have already been defined, because when $p=qr$ and
$q\neq p$, $r\neq p$, the length of the paths $q$ and $r$ is
strictly less than the length of $p$. We define $f_1(p)$ and
$g_2(p)$ to be any elements of $K$, and then since either $f_1(u)=1$
or $g_2(v)=1$ (since either $u\notin S$ or $v\notin S$), we can
choose suitable $g_1(p)$ and $f_2(p)$ such that (\ref{eq1}) holds
true.\\ The fact that $C$ is coreflexive if and only if so is $C_0$
follows now follows directly from Theorem
\ref{thgeneralcoreflexive}.
\end{proof}

Moreover, we can extend the result in the previous proposition to
subcoalgebras of path coalgebras.

\begin{proposition}\label{c.ref}
Let $C$ be a subcoalgebra of a path coalgebra $K\Gamma$, such that
there are only finitely many paths between any two vertices in
$\Gamma$ . Then $C$ is coreflexive if and only if $C_0$ is
coreflexive.
\end{proposition}
\begin{proof}
Let $\Gamma'$ be the subquiver of $\Gamma$ whose vertices are all
the vertices $v$ of $\Gamma$ such that there is an element
$c=\sum_i\alpha_ip_i\in C$, where the $\alpha_i$'s are nonzero
scalars and the $p_i$'s are pairwise distinct paths, and at least
one $p_i$ passes through $v$. The arrows of $\Gamma'$ are all the
arrows of $\Gamma$ between vertices of $\Gamma'$. Clearly, there are
only finitely many paths between any two vertices in $\Gamma'$. Then
we have that $C$ is a subcoalgebra of $K\Gamma'$ and
$C_0=(K\Gamma')_0$. Obviously, $C_0\subset (K\Gamma')_0$; for the
converse, let us consider a vertex $u$ in $\Gamma'$, so there is
$c\in C$ such that $c=\sum\limits_i\alpha_ip_i$, with $\alpha_i\neq
0$ and distinct $p_i$'s, and some $p_k$ passes through $u$. Let us
write then $p_k=qr$ such that $q$ ends at $u$ and $r$ begins at $u$.
Since $C$ is a subcoalgebra of $K\Gamma'$ it is also a
sub-bicomodule, so then $r^*cq^*\in C$, where $q^*,r^*\in
(K\Gamma')^*$ are equal to $1$ on $q,r$ respectively and $0$ on all
other paths of $K\Gamma'$. Now
$$r^*p_iq^*=\sum\limits_{p_i=stw}q^*(s)tr^*(w)$$ and the only non-zero terms
can occur if $p_i=qt_ir$, where $t_i$ is a path starting and ending
at $u$ (loop at $u$). Let $J$ be the set of these indices. In this
situation $r^*p_iq^*=t_i$. Note that since the $p_i$'s are distinct,
the $t_j$'s, $j\in J$ are distinct too. Also, since $p_k=qr$, there
is at least such a $j$. We have $r^*cq^*=\sum\limits_j\alpha_jt_j$,
with all $t_j$ beginning and ending at $u$, and $t_k=u$. Let $l\in
J$ be an index such that $t_l$ has maximum length among the $t_j$'s,
$j\in J$. We note then that $t_l^*t_j=0$ if $j\neq l$, since for any
decomposition $t_j=st$, we have $t\neq t_l$ because of the
maximality of $t_l$ and of the fact that $t_j\neq t_l$. However,
$t_l^*t_l=u$. Therefore, $t_l^*c=\alpha_lu\in C$, so $u\in C$
since $\alpha_l\neq 0$. \\
Thus if $C_0$ is coreflexive, we have that $(K\Gamma')_0$ is
coreflexive, and then by Lemma \ref{lemaconditii}, we have that
$K\Gamma'$ is coreflexive. Then $C$ is coreflexive, as a
subcoalgebra of  $K\Gamma'$. Conversely, if $C$ is coreflexive, then
clearly $C_0$ is coreflexive.
\end{proof}

\begin{corollary}\label{c.incoref}
Let $C$  be a subcoalgebra of an incidence coalgebra $KX$. Then $C$
is coreflexive if and only if $C_0$ is coreflexive.
\end{corollary}
\begin{proof}
As explained in Section \ref{sectincidence}, $KX$ can be embedded in
a path coalgebra $K\Gamma$, where $\Gamma$ is a quiver for which
there are finitely many paths between any two vertices. Then $C$ is
isomorphic to a subcoalgebra of  $K\Gamma$ and we apply Proposition
\ref{c.ref}.
\end{proof}

Recall that for a path coalgebra or incidence coalgebra $C$,
$C_0\sim K^{(X)}$, where $X$ is the set of grouplike elements in
$C$. At this point, we believe it is worth mentioning that by
\cite[Theorem 3.7.3]{HR}, $K^{(X)}$ is coreflexive whenever $X$ is a
nonmeasurable cardinal. More precisely, an ultrafilter ${\mathcal
F}$ on a set $X$ is called an {\it Ulam} ultrafilter if ${\mathcal
F}$ is closed under countable intersection. $X$ is called {\it
nonmeasurable} (or {\it reasonable} in the language of \cite{HR}) if
every Ulam ultrafilter is principal (i.e., it equals the collection
of all subsets of $X$ containing some fixed $x\in X$). The class of
nonmeasurable sets contains the countable sets and is closed under
usual set-theoretic constructions, such as the power set, subsets,
products, and unions. If a non-reasonable (i.e. measurable) set
exists, its cardinality has to be ``very large" (i.e., inaccessible
in the sense of set theory).

\vspace{.4cm}

We now give an example to show that it is possible to have a
coalgebra which is both coreflexive, and satisfies the path
coalgebra ``recovery" conditions of Theorem \ref{t.pathalgco};
however, in its quiver, some vertices are joined by infinitely many
paths. Thus, in general, the coreflexivity question for path
coalgebras is more complicated.

\begin{example}
Consider the path coalgebra $C$ of the following quiver $\Gamma$:
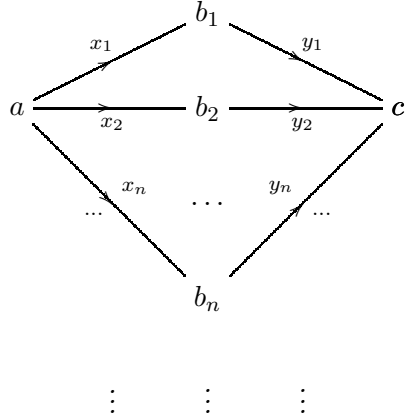
$$\xygraph{ 
!~:{@{-}|@{>}}
a(:@/^0pc/[urr]{b_1}^{x_{1}}:@/^0pc/[drr]{c}^{y_{1}}"a",
:@/^0pc/[rr]{b_2}_{x_{2}}:@/^0pc/[rr]{c}_{y_{2}}[dll]{\dots}"a",
:@/^0pc/[ddrr]{b_n}^{x_{n}}_\dots:@/^0pc/[uurr]{c}^{y_{n}}_\dots
"a",
[dddr]{\vdots}[r]{\vdots}[r]{\vdots}
 )}$$ Here there are infinitely many vertices $b_n$, one for each positive integer $n$.
 Let $W_n$ be the
finite dimensional subcoalgebra of $C$ with basis
$$B=\{a,c,b_1,\dots,b_n,x_1,\dots,x_n,y_1,\dots,y_n,x_1y_1,\dots,x_ny_n\}.$$
We show that $W_n^\perp=W_n^\perp\cdot W_n^\perp$. Let $f\in
W_n^\perp$. We show that we can find $g_1,g_2,h_1,h_2\in W_n^\perp$
such that $f=g_1h_1+g_2h_2$. This condition is already true on
elements of $B$ if we set $g_1,g_2,h_1,h_2$ to be zero on $W_n$. For
$k>n$ we define $g_1,g_2,h_1,h_2$ on $x_k,y_k$ and $x_ky_k$ such
that
\begin{eqnarray*}
f(x_ky_k) & = & \sum\limits_{i=1,2}(g_i(a)h_i(x_ky_k)+g_i(x_k)h_i(y_k)+g_i(x_ky_k)h_i(c))\\
f(x_k) & = & \sum\limits_{i=1,2}(g_i(a)h_i(x_k)+g_i(x_k)h_i(b_k))\\
f(y_k) & = & \sum\limits_{i=1,2}(g_i(b_k)h_i(y_k)+g_i(y_k)h_i(c))\\
f(b_k) & = & \sum\limits_{i=1,2}g_i(b_k)h_i(b_k)\\
\end{eqnarray*}
and since $g_i(a)=h_i(a)=g_i(c)=h_i(c)=0$ this is equivalent to the
matrix equality
$$
\left(\begin{array}{cc}
    f(b_k) & f(y_k) \\
    f(x_k) & f(x_ky_k)
\end{array}\right)
= \left(\begin{array}{c} g_1(b_k)\\ g_1(x_k) \end{array}\right)
\cdot \left(\begin{array}{cc} h_1(b_k) & h_1(y_k)
\end{array}\right)+ \left(\begin{array}{c} g_2(b_k)\\ g_2(x_k)
\end{array}\right) \cdot \left(\begin{array}{cc} h_2(b_k) & h_2(y_k)
\end{array}\right)
$$
But it is a standard linear algebra fact that any arbitrary $2\times
2$ matrix can be written this way as a sum of two matrices of rank
1, and thus the claim is proved. Since every finite dimensional
subcoalgebra $V$ of $C$ is contained in some $W_n$ with
$W_n^\perp=W_n^\perp\cdot W_n^\perp$ and $C_0\cong K^{\ZZ_{> 0}}$ is
coreflexive, by Theorem \ref{thgeneralcoreflexive} we obtain that
$C$ is coreflexive.
\end{example}


\subsection*{Reflexivity for quiver and incidence algebras}

Recall from \cite{T1} that an algebra is called reflexive if the
natural (evaluation) map from $A$ to $A^{0*}$ is an isomorphism.
Using our construction in Section 2, we can extend this to algebras
with enough idempotents, and call such an algebra {\it reflexive} if
the map $\Phi:a\longmapsto (f\mapsto f(a))\in A^{0*}$ is an
isomorphism. We note that in general the coalgebra $A^0$ is a
coalgebra with counit, and therefore, $A^{0*}$ is an algebra with
unit. Hence, a reflexive algebra must be unital. Parallel to
algebras with unit we call an algebra {\it proper} if the map $\Phi$
is injective and we call $A$ {\it weakly reflexive} if $\Phi$ is
surjective. It is not difficult to see that an algebra is proper if
and only if the intersection of all cofinite ideals is $0$.

\begin{theorem}\label{t.5.7}
Let $\Gamma$ be a quiver. Then:\\
(i) The quiver algebra $K[\Gamma]$ is proper.\\
(ii) $K[\Gamma]$ is reflexive (equivalently, weakly reflexive) if
and only if it is finite dimensional, equivalently, $\Gamma$ has
finitely many vertices and arrows, and has no oriented cycles.
\end{theorem}
\begin{proof}
(i) follows since $K[\Gamma]$ embeds in $(K\Gamma)^*$ which is
proper by \cite[Proposition 3.1]{T1}, and one can easily see that
\cite[Proposition 3.4]{T1} stating that a subalgebra of a proper
algebra is proper can be extended to algebras with enough idempotents.
Alternatively, one can see that the
intersection of cofinite ideals of $K[\Gamma]$ is always $0$.\\
(ii) Assume $K[\Gamma]$ is weakly reflexive, so
$K[\Gamma]\rightarrow K[\Gamma]^{0*}$ is surjective. The inclusion
$K\Gamma\subseteq K[\Gamma]^0$ yields a surjective morphism of
algebras $K[\Gamma]^{0*}\rightarrow (K\Gamma)^*$. This shows that
the natural map $\psi:K[\Gamma]\hookrightarrow (K\Gamma)^*$ is
surjective (and, in fact, bijective). Consider the ``gamma function"
on $K[\Gamma]$, i.e. the function $\gamma\in K[\Gamma]$ equal to $1$
on all paths. Then $\gamma$ is in the image of $\psi$, and since
every function in the image of $\psi$ has finite support as a
function on the set of paths of $\Gamma$, it follows that there are
only finitely many paths in $\Gamma$. Therefore, $K[\Gamma]$ is
finite dimensional. The converse is obvious (as noticed before).
\end{proof}


In the case of incidence algebras, using \cite[Proposition 6.1]{T1}
which states that a coalgebra $C$ is coreflexive if and only if
$C^*$ is reflexive, and using also Corollary \ref{c.incoref}, we
immediately get the following

\begin{theorem}
Let $X$ be a locally finite partially ordered set. Then the following assertions
are equivalent:\\
(i) The incidence algebra $IA(X)$ of $X$ over $K$ is reflexive.\\
(ii) The incidence coalgebra $KX$ is coreflexive. \\
(iii) The coalgebra $(KX)_0=KX_0$ (the grouplike coalgebra on the elements of $X$) is coreflexive.\\
(iv) The algebra $K^X$ of functions on $X$ is reflexive.
\end{theorem}

These yield as a corollary the algebra analogue of Proposition
\ref{p.IPcoalg}.

\begin{corollary}
Let $A$ be an algebra of a non-measurable cardinality. Then $A$ is
isomorphic both to a quiver algebra and to an incidence algebra if
and only if and only if it is the quiver algebra of a finite quiver
with no oriented cycles, equivalently, it is elementary, finite
dimensional and hereditary.
\end{corollary}
\begin{proof}
If $A\cong K[\Gamma]\cong IA(X)$ for a quiver $\Gamma$ and a locally
finite partially ordered set $X$, then $K^{(X)}$ is coreflexive by \cite{HR} since
$X$ is also non-measurable. Now $A\cong IA(X)$ is reflexive since
$K^X\cong (K^{(X)})^*$ is reflexive by \cite[Proposition 6.1]{T1}.
By Theorem \ref{t.5.7}, $A\cong K[\Gamma]$ must be finite
dimensional since it is reflexive.  The final statements follow from
the well known characterizations of finite dimensional quiver
algebras.
\end{proof}




\subsection*{An application}

We give now an application of our considerations on coreflexive
coalgebras. If $\Gamma,\Gamma'$ are quivers, then we consider the
quiver $\Gamma\times\Gamma'$ defined as follows. The vertices are
all pairs $(a,a')$ for $a,a'$ vertices in $\Gamma$ and $\Gamma'$
respectively. The arrows are the pairs $(a,x')$, which is an arrow
from $(a,a_1')$ to $(a,a_2')$, where $a$ is a vertex in $\Gamma$ and
$x'$ is an arrow from $a_1'$ to $a_2'$ in $\Gamma'$, and the pairs
$(x,a')$, which is an arrow from $(a_1,a')$ to $(a_2,a')$, where $x$
is an arrow from $a_1$ to $a_2$ in $\Gamma$, and $a'$ is a vertex in
$\Gamma'$. Let $p=x_1x_2\dots x_n$ be a path in $\Gamma$ going (in
order) through the vertices $a_0,a_1,\dots,a_n$ and $q=y_1y_2\dots
y_k$ be a path in $\Gamma'$ going through vertices
$b_0,b_1,\dots,b_k$ (some vertices may repeat). We consider the 2
dimensional lattice $L=\{0,\ldots ,n\}\times \{0,\ldots ,k\}$. A
lattice walk is a sequence of elements of $L$ starting with $(0,0)$
and ending with $(n,k)$, and always going either one step to the
right or one step upwards in $L$, i.e., $(i,j)$ is followed either
by $(i+1,j)$ or by $(i,j+1)$. There are ${n+k}\choose {k}$ such
walks.\\
To $p,q$ and a lattice walk
$(0,0)=(i_0,j_0),(i_1,j_1),\ldots,(i_{n+k},j_{n+k})=(n,k)$ in $L$ we
associate a path of length $n+k$ in $\Gamma\times \Gamma'$, starting
at $(a_0,b_0)$ and ending at $(a_n,b_k)$ such that the $r$-th arrow
of the path, from $(a_{i_{r-1}},b_{j_{r-1}})$ to $(a_{i_r},b_{j_r})$
is $(x_{r-1},b_{j_{r-1}})$ if $i_r=i_{r-1}+1$, and
$(a_{i_{r-1}},y_{r-1})$ if $j_r=j_{r-1}+1$. \\
Conversely, if $\gamma$ is a path in $\Gamma\times \Gamma'$, there
are (uniquely determined) paths $p$ in $\Gamma$ and $q$ in
$\Gamma'$, and a lattice walk such that $\gamma$ is associated to
$p,q$ and that lattice walk as above. Indeed, we take $p$ to be the
path in $\Gamma$ formed by considering the arrows $x$ such that
there are arrows of the form $(x,a')$ in $\gamma$, taken in the
order they appear in $\gamma$. Similarly, $q$ is formed by
considering the arrows of the form $(a,y)$ in $\gamma$. The lattice
walk is defined according to the succession of arrows in $\gamma$.\\
 For two such paths
$p,q$ let us denote $W(p,q)$ the set of all paths in $\Gamma\times
\Gamma'$ associated to $p$ and $q$ via lattice walks.

\vspace{.4cm}

{\bf Functoriality, (co)products of quivers and recovery problems}.
We note that if $\Gamma$ and $\Gamma'$ satisfy condition (i) in
Theorem 3.3 (i.e. if their path coalgebras can be recovered as
finite duals of the corresponding quiver algebras), then $\Gamma
\times \Gamma'$ satisfies this condition, too. Indeed, the
description of the arrows in $\Gamma\times \Gamma'$ shows that there
are finitely many arrows between any two vertices. Also, if an
oriented cycle existed in $\Gamma\times \Gamma'$, then it would
produce an oriented cycle in each of $\Gamma$ and $\Gamma'$.

Also, if $\Gamma$ and $\Gamma'$ satisfy condition (iii) in Theorem
\ref{thpathsemiperfect} (i.e. if their quiver algebras can be
recovered as the rational part of the dual of the corresponding path
coalgebras), then $\Gamma \times \Gamma'$ satisfies this condition,
too. Indeed, a path in $\Gamma \times \Gamma'$ starting at the
vertex $(a,a')$ is determined by a path in $\Gamma$ starting at $a$,
a path in $\Gamma'$ starting at $a'$ (and there are finitely many
such paths in both cases), and a lattice walk (chosen from a finite
family). These can be extended to finite products of quivers in the
obvious way.

Given a family of quivers $(\Gamma_i)_i$, one can consider the
coproduct quiver $\Gamma=\coprod\limits_i\Gamma_i$. The path
coalgebra functor commutes with coproducts and one has
$K\Gamma=\bigoplus\limits_i K\Gamma_i$. Also, the quiver algebra
functor from the category of quivers to the category of algebras
with enough idempotents has the property that
$K[\coprod\limits_i\Gamma_i]=\bigoplus\limits_i K[\Gamma_i]$. It is
clear that $\coprod\limits_i\Gamma_i$ satisfies the conditions of
Theorem \ref{t.pathalgco} (i) if and only if each $\Gamma_i$
satisfies the same condition, so each $K\Gamma_i$ can be recovered
from $K[\Gamma_i]$ if and only if $K\Gamma$ is recoverable from
$K[\Gamma]$. Also, each of the quivers $(\Gamma_i)_i$ satisfies
condition (iii) in Theorem \ref{thpathsemiperfect}, if and only if
so does their disjoint union $\coprod\limits_i\Gamma_i$. In
coalgebra terms, this is justified by the fact that a direct sum
$\bigoplus\limits_iC_i$ of coalgebras is semiperfect if and only if
each $C_i$ is semiperfect.

\vspace{.4cm}

Returning to coreflexivity problems, we need the following

\begin{lemma}\label{prodquiver}
The linear map $\alpha:K\Gamma\otimes K\Gamma'\hookrightarrow
K(\Gamma\times\Gamma')$ defined as $\alpha(p\otimes
q)=\sum\limits_{w\in W(p,q)}w$ for $p\in\Gamma$, $q\in\Gamma'$
paths, is an injective morphism of $K$-coalgebras.
\end{lemma}
\begin{proof}
We keep the notations above. Denote $\delta$ and $\Delta$ the
comultiplications of $K\Gamma\otimes K\Gamma'$ and
$K(\Gamma\times\Gamma')$. We have that
\begin{eqnarray*}
 \delta\alpha (p\otimes
q)&=&\sum_{w\in W(p,q)}\sum_{w'w''=w}w'\otimes w''\\
(\alpha \otimes \alpha)\Delta (p\otimes
q)&=&\sum_{p'p''=p}\sum_{q'q''=q}\sum_{u\in W(p',q')\atop v\in
W(p'',q'')}u\otimes v \end{eqnarray*}

On one hand, if $p=p'p''$, $q=q'q''$, $u\in W(p',q')$ and $v\in
W(p'',q'')$, we have that $uv\in W(p,q)$. On the other hand, if
$w\in W(p,q)$ and $w=w'w''$, then there exist $p'p''$ in $\Gamma$
and $q',q''$ in $\Gamma'$ such that $p=p'p''$, $q=q'q''$, $w'\in
W(p',q')$ and $w''\in W(p'',q'')$. These show that $\delta\alpha
(p\otimes q)=(\alpha \otimes \alpha)\Delta (p\otimes q)$, i.e.,
$\alpha$ is a morphism of coalgebras (the compatibility with counits
is easily
verified). \\
To prove injectivity, if $p=x_1x_2\ldots x_n$ is a path in $\Gamma$
starting at $a_0$ and ending at $a_n$, and $q=y_1y_2\ldots y_k$ is a
path in $\Gamma'$ starting at $b_0$ and ending at $b_k$, we denote
by $(p^*,q^*)$ the linear map on $K(\Gamma\times\Gamma')$ which
equals $1$ on the path
$(x_1,b_0),\dots,(x_n,b_0),(a_n,y_1),\dots,(a_n,y_k)$ (for
simplicity we also denote this path by $(p,b_0);(a_n,q)$) and $0$ on
the rest of the paths. Let $\sum_i\lambda_ip_i\otimes q_i\in \Ker
(\alpha)$. Then we have that \begin{equation}
\label{ecinjal}\sum_i\sum_{w\in
W(p_i,q_i)}\lambda_iw=0\end{equation} Fix some $j$. Say that $p_j$
ends at $a_n$ and $q_j$ starts at $b_0$. We have that
$$(p_j^*,q_j^*)(w)=\left\{
\begin{array}{l}
0,\mbox{ if }w\in W(p_i,q_i),i\neq j\\
0, \mbox{ if }w\in W(p_j,q_j)\mbox{ and }w\neq (p_j,b_0),(a_n,q_j)\\
1, \mbox{ if }w=(p_j,b_0),(a_n,q_j)
\end{array}\right.$$
Note that we used the fact that $W(p,q)\cap W(p',q')=\emptyset$ for
$(p,q)\neq (p',q')$. Now applying $(p_j^*,q_j^*)$ to (\ref{ecinjal})
we see that $\lambda_j=0$. We conclude that $\alpha$ is injective.
\end{proof}

Combining the above, we derive a result about tensor products of
certain coreflexive coalgebras. It is known that a tensor product of
a coreflexive and a strongly coreflexive coalgebra is coreflexive
(see \cite{Rad}; see also \cite{T2}). It is not known whether the
tensor product of coreflexive coalgebras is necessarily coreflexive.
We have the following consequences.

\begin{proposition}\label{CxD2}
Let $C,D$ be co-reflexive subcoalgebras of path coalgebras $K\Gamma$
and $K\Gamma'$ respectively such that between any two vertices in
$\Gamma$ and $\Gamma'$ respectively there are only finitely many
paths. Then $C\otimes D$ is co-reflexive.
\end{proposition}
\begin{proof}
Without any loss of generality we may assume that
$C_0=(K\Gamma)_0=K^{(\Gamma_0)}$ and
$D_0=(K\Gamma')_0=K^{(\Gamma_0)}$ (otherwise we replace $\Gamma$ and
$\Gamma'$ by appropriate subquivers), where $K^{(\Gamma_0)}$ denotes
the grouplike coalgebra with basis the set $\Gamma_0$ of vertices of
$\Gamma$. Now $C\otimes D$ is a subcoalgebra of $K\Gamma\otimes
K\Gamma'$, so by Lemma \ref{prodquiver}, it also embeds in $K(\Gamma
\times \Gamma')$. Since the coradical of $K(\Gamma \times \Gamma')$
is $K^{(\Gamma_0\times \Gamma'_0)}$, and $K^{(\Gamma_0\times
\Gamma'_0)}\simeq K^{(\Gamma_0)}\otimes K^{(\Gamma'_0)}=C_0\otimes
D_0\subseteq C\otimes D$, we must have that $(C\otimes
D)_0=K^{(\Gamma_0\times \Gamma'_0)}$. We claim that
$K^{(\Gamma_0\times \Gamma'_0)}$ is coreflexive. Indeed, this is
obvious if $\Gamma_0$ and $\Gamma'_0$ are both finite. Otherwise,
${\rm card}(\Gamma_0\times \Gamma'_0)=\max\{{\rm
card}(\Gamma_0),{\rm card}(\Gamma'_0)\}$, hence $K^{(\Gamma_0\times
\Gamma'_0)}$ is isomorphic either to $K^{(\Gamma_0)}$ or to
$K^{(\Gamma'_0)}$, both of which are coreflexive by Proposition
\ref{c.ref}. Since it is clear that in $\Gamma\times \Gamma'$ there
are also finitely many paths between any two vertices, we can use
Proposition \ref{c.ref} to show that $C\otimes D$ is coreflexive.
\end{proof}

\begin{corollary}\label{CxD1}
If $C,D$ are coreflexive subcoalgebras of incidence coalgebras, then
$C\otimes D$ is coreflexive.
\end{corollary}
\begin{proof}
It follows immediately from the embedding of $C$ and $D$ in path
coalgebras verifying the hypothesis of Proposition \ref{CxD2}.
\end{proof}


\bigskip\bigskip\bigskip

\begin{center}
\sc Acknowledgment
\end{center}
The authors would like to greatly acknowledge the very careful
reading and detailed comments of the referee, which led to many
improvements mathematics and exposition wise. In particular, we are
in debt for questions and suggestions which prompted the expansion
of Sections 3, 4 and 5 with several new results.

The research of this paper was partially supported by the UEFISCDI
Grant PN-II-ID-PCE-2011-3-0635, contract no. 253/5.10.2011 of
CNCSIS. For the second author, this work was supported by the
strategic grant POSDRU/89/1.5/S/58852, Project ``Postdoctoral
program for training scientific researchers'' cofinanced by the
European Social Fund within the Sectorial Operational Program Human
Resources Development
2007-2013. 

\end{document}